 \documentclass[11pt]{amsart}
\usepackage{mathtools}

\mathtoolsset{showonlyrefs=true}

\usepackage[a4paper,left=35mm,right=35mm,top=30mm,bottom=30mm,marginpar=25mm]{geometry}

\usepackage{amsmath}
\usepackage{amssymb}
\usepackage{amsthm}
\usepackage{eurosym}
\usepackage[dvips]{graphics}
\usepackage{graphicx}
\usepackage{epsfig}
\usepackage[colorlinks]{hyperref}

\usepackage{dsfont}
\usepackage[nocompress, space]{cite}

\usepackage[displaymath,mathlines]{lineno}
\usepackage{comment}

\allowdisplaybreaks

\usepackage{hyperref}

\usepackage{ifthen}

\makeindex

\newcommand{\sign}{\operatorname{sign}}

\renewcommand{\div}{\operatorname{div}}

\newcommand{\Rr}{{\mathbb{R}}}

\newcommand{\Tt}{{\mathbb{T}}}

\newcommand{\Hh}{{\overline{H}}}

\newcommand{\Aa}{{\mathcal{A}}}

\newcommand{\epsi}{\varepsilon}
\def\d{{\rm d}}
\def\dx{{\rm d}x}

\def\dt{{\rm d}t}
\def\leq{\leqslant}
\def\geq{\geqslant}

\numberwithin{equation}{section}

\newtheoremstyle{thmlemcorr}{10pt}{10pt}{\itshape}{}{\bfseries}{.}{10pt}{{\thmname{#1}\thmnumber{
#2}\thmnote{ (#3)}}}
\newtheoremstyle{thmlemcorr*}{10pt}{10pt}{\itshape}{}{\bfseries}{.}\newline{{\thmname{#1}\thmnumber{
#2}\thmnote{ (#3)}}}
\newtheoremstyle{defi}{10pt}{10pt}{\itshape}{}{\bfseries}{.}{10pt}{{\thmname{#1}\thmnumber{
#2}\thmnote{ (#3)}}}
\newtheoremstyle{remexample}{10pt}{10pt}{}{}{\bfseries}{.}{10pt}{{\thmname{#1}\thmnumber{
#2}\thmnote{ (#3)}}}
\newtheoremstyle{ass}{10pt}{10pt}{}{}{\bfseries}{.}{10pt}{{\thmname{#1}\thmnumber{
A#2}\thmnote{ (#3)}}}

\theoremstyle{thmlemcorr}
\newtheorem{theorem}{Theorem}
\numberwithin{theorem}{section}
\newtheorem{lemma}[theorem]{Lemma}
\newtheorem{corollary}[theorem]{Corollary}
\newtheorem{proposition}[theorem]{Proposition}

\theoremstyle{thmlemcorr*}
\newtheorem{theorem*}{Theorem}
\newtheorem{lemma*}[theorem]{Lemma}
\newtheorem{corollary*}[theorem]{Corollary}
\newtheorem{proposition*}[theorem]{Proposition}
\newtheorem{problem*}[theorem]{Problem}
\newtheorem{conjecture*}[theorem]{Conjecture}

\theoremstyle{defi}
\newtheorem{definition}[theorem]{Definition}
\newtheorem{hyp}{Assumption}

\newtheorem{problem}{Problem}

\theoremstyle{remexample}
\newtheorem{remark}[theorem]{Remark}

\newtheorem{teo}[theorem]{Theorem}
\newtheorem{lem}[theorem]{Lemma}
\newtheorem{pro}[theorem]{Proposition}
\newtheorem{cor}[theorem]{Corollary}

\theoremstyle{ass}%

\theoremstyle{definition} % Mette testo teoremi etc NON corsivo

\begin{document}

\title{$C^{1,\alpha}$ Regularity For Stationary Mean-Field Games With Logarithmic Coupling}

%\thanks{The research reported in this publication was supported by funding from King Abdullah University of Science and Technology (KAUST) including baseline funds and KAUST OSR-CRG2021-4674.}
\author{Tigran Bakaryan}
\address[T. Bakaryan]{
	King Abdullah University of Science and Technology (KAUST), CEMSE Division, Thuwal 23955-6900. Saudi Arabia.}
\email{tigran.bakaryan@kaust.edu.sa}
\author{Giuseppe Di Fazio}
\address[G. Di Fazio]{
	University of Catania, Italy.}
\email{giuseppedifazio@unict.it}
\author{Diogo A. Gomes}
\address[D. A. Gomes]{
	King Abdullah University of Science and Technology (KAUST), CEMSE Division, Thuwal 23955-6900. Saudi Arabia.}
\email{diogo.gomes@kaust.edu.sa}

 %Existence problems: global existence, local existence, non-existence

\thanks{The authors were supported by King Abdullah University of Science and Technology (KAUST) baseline funds and KAUST OSR-CRG2021-4674.
}
\date{\today}

\maketitle

\begin{abstract}
This paper investigates stationary mean-field games (MFGs) on the torus with Lipschitz non-homogeneous diffusion and logarithmic-like couplings. The primary objective is to understand the existence of $C^{1,\alpha}$ solutions to address the research gap between low-regularity results for bounded and measurable diffusions and the smooth results modeled by the Laplacian.

We use the Hopf--Cole transformation to convert the MFG system into a scalar elliptic equation. Then, 
we apply Morrey space methods to establish existence and regularity of solutions. The introduction of Morrey space methods offers a novel approach to address regularity issues in the context of MFGs. 
\end{abstract}

\section{Introduction}

This paper studies stationary mean-field games (MFGs) on the torus, focusing on non-homogeneous diffusion and logarithmic-like couplings. 
MFGs offer a framework for analyzing large populations of competing rational agents. These games have two primary components: a Hamilton-Jacobi equation that governs each agent's value function and a Fokker-Planck equation that describes the evolution of agent density.

We consider a stationary MFG with a non-homogeneous diffusion matrix, \(A(x)\), and a logarithmic coupling, \(g(\log m)\), between the equations. More precisely, the problem we are investigating is the following.
\begin{problem}\label{problem-1} 
	Let \(A(x)=\left(a^{i,j}(x) \right)_{i,j}\) be a non negative definite \(d\times d\) matrix-valued function defined on the $d$-dimensional torus \( \Tt^d\). Let $V$ and $g$ be given continuous functions in $\Tt^d$. 
	Find \((u,m)\in (C^2(\Tt^d))^2\) and \(\bar{H}\in \Rr\) satisfying  
	\begin{equation}\label{problem-1_eq1}
		\begin{cases}
			-\div(A Du^T) +\frac{1}{2}DuADu^T+V(x) =g\left( \log m\right) +\Hh\\
			-\div(A Dm^T)  -\div(m A Du^T) =0                                    \\
			m\geq 0,\quad\int_{\Tt^d}m\dx\,=1.
		\end{cases}\quad x\in\Tt^d
	\end{equation}
\end{problem}
The existence of 
weak solutions for this problem in the sense of monotone operators can be proved 
with minimal assumptions on the matrix, $A$, see \cite{FGT1, FG2}. 
On the higher end of the regularity spectrum, when \( A \) is the identity matrix and the diffusion corresponds to the Laplacian, smooth solutions for this and related problems were studied in \cite{GM}, \cite{MR4175148}, \cite{EFGNV2017}, and \cite{PV15}.
Non-homogeneous diffusion can arise in many applications, prompting our investigation into the effects of replacing the Laplacian with a more general elliptic operator. However, 
the presence of the non-constant matrix \(A(x)\) in the diffusion terms introduces analytical challenges. 
Some of these challenges were previously addressed in the literature. 
For example, the 
hypoelliptic 
case was examined in \cite{MR4132070} and  \cite{DragoniFeleqi2018}. The primary challenge in the
hypoelliptic case is degeneracy. In this work, we address another difficulty: the lower regularity of the coefficients
of $A$, more precisely assuming only Lipschitz continuity. In \cite{CLLP}, the Laplacian case was addressed using the Hopf-Cole for the time-dependent problem. Some of the estimates in that paper would also be valid for a parabolic version of 
Problem \ref{problem-1}.
Later,   elliptic MFGs with bounded coefficients were considered in \cite{bocorsporr}, where a comprehensive theory for solutions in Lebesgue spaces was developed. 
This paper addresses the research gap between the low-regularity results for bounded and measurable \( A \) and the smooth results modeled by the Laplacian. 
In particular, 
we examine the case where $A$ is Lipschitz continuous. 
The precise set of assumptions under which our analysis is conducted is  outlined  in Section \ref{section:Main_Assumptions}.  
The structure of Problem \ref{problem-1} is particularly suitable for the Hopf-Cole transformation, as we explain below. 
Choosing models that allow the use of the Hopf-Cole transformation simplifies the analysis and may serve as a guide to more complex problems. The ultimate goal is to develop mathematical techniques that can be applied to more general MFGs than previously explored.

We consider the following definition of solution for Problem \ref{problem-1}.
\begin{definition}\label{def-weak} \rm 
	A triple \((u,m,\bar{H})\in C^1(\Tt^d)\times C^1(\Tt^d)\times\Rr\) is a weak solution to Problem \ref{problem-1} if the following conditions hold: \(m>0,\quad\int_{\Tt^d}m\dx\,=1\) and equations in \eqref{problem-1_eq1} hold in the sense of distributions i.e.
	
	\begin{equation}\label{def-weak-HJ-eq}
		\int_{\Tt^d}D\varphi A Du^T +\left( \frac{1}{2}DuADu^T +V(x)-\bar{H}\right) \varphi\dx\,=\int_{\Tt^d}g\left( \log m\right)\varphi\dx
	\end{equation}
	and
	\begin{equation}\label{def-weak-FP-eq}
		\int_{\Tt^d}D\varphi A Dm^T  +m D\varphi  A Du^T  \dx=0,
	\end{equation}
	for any $\varphi$ smooth in $\Tt^d$ and - by density - for any \(\varphi\in H^1(\Tt^d) \).
\end{definition}

The Hopf--Cole transformation, a cornerstone in our study, transforms viscous quadratic Hamilton-Jacobi equations into linear equations. 
For instance, consider a quadratic Hamilton-Jacobi equation
\[
-\Delta u+|Du|^2+V(x)=0.
\]
Under the Hopf--Cole transformation
$\displaystyle{
	v=e^{-u}
}$
simplifies to 
\[
-\Delta v +V(x) v=0.
\]
This transformation has proven invaluable in the MFG setting, especially for quadratic Hamiltonians and logarithmic coupling terms  (see, e.g., \cite{GueantT},\cite{GueU},\cite{gueant10},\cite{cirant3}). 
For the MFG in Problem \ref{problem-1}, 
the Hopf-Cole transform eliminates the second equation in the MFG system; that is, 
\[
-\div(A Dm^T)  -\div(m A Du^T)=-\div( A (Dm^T +mDu^T))=0, 
\]
when $m=e^{-u}$. 
Thus, 
we reduce \eqref{problem-1_eq1} to 
\begin{equation}\label{eq-HJ_Geop-after-Hopf}
	-\div(A Du^T) +\frac{1}{2}DuADu^T+V(x) -\bar{H}-g(-u)=0,
\end{equation}
which is a scalar elliptic equation. In  Section \ref{sec:4},  borrowing ideas from \cite{Boccardo1988ExistenceOB}, we prove existence and H{\"o}lder continuity  of solutions to   \eqref{eq-HJ_Geop-after-Hopf}.  Further regularity of solutions to equation in \eqref{eq-HJ_Geop-after-Hopf} explored in Section \ref{sec:5}. More precisely, 
  we employ techniques from the study of elliptic equations in Morrey spaces, in \cite{DiFazio_2020_Ind}, \cite{DiFazio1993}, and \cite{DiFazio2008harnack} to obtain      \(C^{1,\alpha}\) regularity the solutions. Building upon these results, Section \ref{sec:6} establishes the existence and uniqueness of solutions to Problem \ref{problem-1}. This is  our main result, stated as follows.
\begin{teo}\label{teo-main}
	Suppose that Assumptions \ref{assump-A_matrix-eliptic}-\ref{assump-g_mon-concave} hold (see Section \ref{section:Main_Assumptions}). Then,  there exists a unique triple \((u,m,\bar{H})\in C^1(\Tt^d)\times C^1(\Tt^d)\times\Rr\) solving Problem \ref{problem-1} in the sense of Definition \ref{def-weak}.
\end{teo}
In \cite{bocorsporr},  the authors examined MFGs under a more general set of assumptions,  showing the existence of weak solutions in Lebesgue spaces. In contrast, our study establishes a higher degree of regularity, specifically $C^{1,\alpha}$ regularity. Understanding the solution regularity is crucial for numerical methods and applications. 

Our proof strategy hinges on new estimates in Morrey spaces and elliptic regularity results. 
To the best of the authors' knowledge, this is the first time Morrey space methods have been used in the context of MFGs. 
Morrey spaces are particularly suitable for studying elliptic equations and are a primary technical tool for obtaining H\"older estimates.
These methods allow us to examine a different range of regularity issues that previous techniques could not address because they work with elliptic equations with limited regularity. 

%The paper is structured as follows: Section \ref{section:Main_Assumptions} introduces the problem formulation and assumptions, while our primary existence and regularity results for unnormalized solutions are presented in Section \ref{sec:4}, borrowing ideas from \cite{Boccardo1988ExistenceOB}.
%In Section \ref{sec:5},  we employ techniques from the study of elliptic equations in Morrey spaces, in \cite{DiFazio_2020_Ind}, \cite{DiFazio1993}, and \cite{DiFazio2008harnack} to improve the regularity of 
%those solutions leading to \(C^{1,\alpha}\) regularity. The existence of solutions for Problem \ref{problem-1} is established in Section \ref{sec:6}.
% 

\section{Main assumptions} \label{section:Main_Assumptions}

Our primary objective is to establish Theorem \ref{teo-main}. This will be accomplished by verifying the existence of a solution to \eqref{eq-HJ_Geop-after-Hopf}, under appropriate assumptions that allow the use of elliptic regularity theory.

The first two assumptions, Assumptions \ref{assump-A_matrix-eliptic} and \ref{assump-V-g_bounded}, provide conditions on the diffusion matrix $A(x)$ and the coupling function $g(u)$, respectively. Assumption \ref{assump-A_matrix-eliptic}  imposes ellipticity and uniform convexity on the Hamiltonian, ensuring it exhibits suitable structure amenable to the techniques used. Assumption  \ref{assump-V-g_bounded}, requires certain growth and monotonicity properties on $g(u)$, allowing $L^\infty$ bounds to be obtained. These assumptions are common in the MFG literature when analyzing regularity issues. All results in Section \ref{sec:4} rely only in these two assumptions. 

The first assumption gives both the ellipticity of the second-order term and the Hamiltonian's uniform convexity, resulting in \eqref{eq-HJ_Geop-after-Hopf} exhibiting uniform ellipticity and convexity in the gradient with a natural growth condition. 

\begin{hyp}\label{assump-A_matrix-eliptic} The matrix \(A(x)=\left(a^{i,j}(x) \right)_{i,j}\) is  uniformly elliptic, i.e., there exists \(\theta_0, \theta_1>0\) such that
	\[
	\theta_0|z|^2\leq zA(x)z^T=\sum_{i,j}^{n}  a_{ij}(x)z_{i}z_{j} \leq \theta_1|z|^2, \qquad (x,z)\in\Tt^d\times\Rr^d.
	\]
\end{hyp}

Assumption \ref{assump-A_matrix-eliptic} is standard as it allows to use elliptic regularity techniques. Two important cases where it is does not hold is the hypoellitic case and the first-order case. In both cases, the techniques used to establish the existence of solutions are quite different from the ones used here. 

The second assumption imposes conditions on $g$, enabling the proof of $L^\infty$ bounds (see  Proposition \ref{pro-Linf_bound-Ell1}). These assumptions are foundational in Section \ref{sec:4} for establishing the existence  of a solution \(u\in H^1(\Tt^d)\cap L^\infty(\Tt^d)\cap C^{0,\alpha}(\Tt^d)\) to \eqref{eq-HJ_Geop-after-Hopf}.
\begin{hyp}\label{assump-V-g_bounded} Let 
	the function, $g:\Rr^d\to\Rr$, is locally Lipschitz continuous and satisfies
	$g(u)\sign(u) \geq C_g|u|-\frac{1}{C_g}$.
\end{hyp}

Assumption \ref{assump-A_matrix_element-bounded}  imposes the Lipschitz continuity of the matrix $A(x)$ and the potential $V(x)$. 
This continuity is critical for the estimates in Morrey spaces and Hölder regularity results. 
No further assumptions are required on \(g\) because the results on Section \ref{sec:4} give that \(u\) is bounded. Thus, it suffices that \(g\) is locally  Lipschitz, as prescribed in Assumption \ref{assump-V-g_bounded}. The Lipschitz regularity 
of these terms allow us to differentiate the equation to bootstrap regularity. This is achieved in 
Section \ref{sec:5} where, we derive the $C^{1, \alpha}$ regularity for the solutions. 

\begin{hyp}\label{assump-A_matrix_element-bounded} 	
	The function \(V\) and the matrix \(A(x)=\left(a^{i,j}(x) \right)_{i,j}\) are Lipschitz continuous, i.e., \(V, a_{ij}\in {Lip}(\Tt^d)\), \(i,j=1,\dots,d\).
\end{hyp}

The last assumption concerns the monotonicity of the coupling function $g(\log m)$. 
This monotonicity  $g(\log m)$ ensures the solution uniqueness through the well-known Lasry-Lions argument. 
\begin{hyp}\label{assump-g_mon-concave} 
	The function \(g(\log(\cdot))\) is 	strictly monotone, i.e., for all $ s_1,s_2\in \Rr^+$
	\[
	(g(\log(s_1))-g(\log(s_2)),s_1-s_2)> 0.
	\]
\end{hyp}
Because $\log m$ is monotone, the preceding assumption is equivalent to the monotonicity of $g$.

\section{Existence of solutions of elliptic equation with quadratic growth} \label{sec:4}

In this section, we prove the existence of weak solutions to non-linear elliptic equations with natural growth in the gradient. We rewrite  \eqref{eq-HJ_Geop-after-Hopf}  as 
\begin{equation}\label{eq-Ell1}
	-\div(A Du^T) +\frac{1}{2}DuADu^T+V_\Hh(x)-g(-u)=0,
\end{equation}
where 
$V_\Hh(x)=V(x)-\bar{H}.$  In Section \ref{ebound}, we consider non-linear
elliptic equations with a bounded non-linearity in the first-order derivatives and prove an existence result. In Section \ref{equad}, using an approximation argument, we get the existence of solutions to \eqref{eq-Ell1}.

\subsection{An elliptic equation with bounded non-linear term}
\label{ebound}

We begin by analyzing elliptic equations of the form
\begin{equation}\label{eq-Ell-bound}
	-\div(A Du^T) +H(x,Du)+V_\Hh(x)-g(-u)=0,
	\qquad x\in\Tt^d,
\end{equation}
where \(H\) is a bounded function and $g$ has linear growth.
We prove the existence of a solution of \eqref{eq-Ell-bound}, which in the next section 
is combined with a limiting argument to obtain a solution for \eqref{eq-Ell1}.

Let \(C_H\) be a positive constant such that
\begin{equation}\label{def-H-bouned}
	|H(x,p)|\leq C_H,
	\qquad  (x,p)\in \Tt^d \times \Rr^d.
\end{equation}
Similarly, 
let $\tilde C_g$ be a positive constant 
such that 
\begin{equation}\label{def-g-bouned}
	|g(s)|\leq \tilde C_g(1+|s|),
	\qquad  s\in \Rr. 
\end{equation}
First, we prove the existence of solutions to \eqref{eq-Ell-bound} assuming, 
in addition to Assumptions \ref{assump-A_matrix-eliptic} and \ref{assump-V-g_bounded}, that 
both 
\eqref{def-H-bouned} and \eqref{def-g-bouned} hold. Then, in Corollary \ref{cor-teo-ex-Linf-ell-bound}, we remove the requirement  \eqref{def-g-bouned}. The 
boundedness of $H$ is handled in the next section. 

\begin{proposition}\label{pro-ex-H1-ell_bounded} Let \(H\)  and $g$ satisfy \eqref{def-H-bouned} and \eqref{def-g-bouned}, respectively. Suppose that Assumptions \ref{assump-A_matrix-eliptic} and \ref{assump-V-g_bounded} hold. 
	Then, there exists \(u\in H^1(\Tt^d)\) solving \eqref{eq-Ell-bound} in the sense of distributions.
\end{proposition}
\begin{proof}
	To prove the existence of solutions to \eqref{eq-Ell-bound}, we use  Leray-Lions theory (see  \cite{LerayLions1965} or \cite[Chapter 5]{MorreyB1966}).
	Consider the operator $\Aa:H^1(\Tt^d)\to H^{-1}(\Tt^d)$ given by
	\begin{equation}\label{def-Leray-Lions-op}
		(\Aa (u),v)=\int_{\Tt^d}Dv A Du^T +(H(x,Du)+V_\Hh(x)-g(-u))v\dx=0,
	\end{equation}
	for all $v\in H^1(\Tt^d)$. 
	Assumptions \ref{assump-A_matrix-eliptic} and \ref{assump-V-g_bounded} ,  \eqref{def-H-bouned}, and
	\eqref{def-g-bouned}
	imply
	\[
	(\Aa (v),v)\geq c ||Dv||^2_{L^2(\Tt^d)}-C\|v\|_{L^2(\Tt^d)}+C_g \|v\|_{L^2(\Tt^d)}^2.
	\]
	Thus, for some suitable constants $\hat c$ and $C$, we have 
	\begin{equation}\label{proof-eq2-pro-ex-H1-ell_bounded}
		\lim\limits_{||v||_{H^1(\Tt^d)}\to \infty}\frac{(\Aa (v),v) }{||v||_{H^1(\Tt^d)}}\geq  \lim\limits_{||v||_{H^1(\Tt^d)}\to \infty}\frac{c||Dv||^2_{L^2(\Tt^d)} +\hat c ||v||^2_{L^2(\Tt^d)}-C}{||v||_{H^1(\Tt^d)}}=+\infty.
	\end{equation}
	Let
	\begin{equation*}
		\Aa^0(x,s,p)=(H(x,p)+V_\Hh(x)-g(-s)), \quad \Aa^1(p)=p A.
	\end{equation*}
	Accordingly, 
	Assumptions \ref{assump-A_matrix-eliptic}, \eqref{def-H-bouned}, and \eqref{def-g-bouned},
	imply
	\begin{equation}\label{proof-eq1-pro-ex-H1-ell_bounded}
		|\Aa^0(x,s,p)| +|\Aa^1(p)|\leq C(1+|s|+|p|), \quad (x,s,p)\in \Tt^d\times\Rr\times \Rr^d.
	\end{equation}
	Using ellipticity of $A$, we get
	\begin{equation}\label{proof-eq1-pro-ex-mono}
		(\Aa^1(p_1)-\Aa^1(p_2))(p_1-p_2)\geq \theta_0 |p_1-p_2|^2, \quad p_1,p_2\in\Rr^d.
	\end{equation}
	Inequalities in  \eqref{proof-eq2-pro-ex-H1-ell_bounded}, \eqref{proof-eq1-pro-ex-H1-ell_bounded} and \eqref{proof-eq1-pro-ex-mono} imply that the operator,  
	\(\Aa\), defined in \eqref{def-Leray-Lions-op} satisfies the conditions of  \cite[Theorem 5.12.2]{MorreyB1966}, which ensures the existence of solutions to  \eqref{eq-Ell-bound}.
\end{proof}
We now prove the boundedness of weak solutions using a truncation argument (see, e.g., \cite{Stampacchia1964quationsED} or \cite{Giachetti_Stamp}).
\begin{teo}\label{teo-ex-Linf-ell-bound}
	Consider the setting of Problem \ref{problem-1}. Let \(H\)  and $g$ satisfy \eqref{def-H-bouned} and \eqref{def-g-bouned}, respectively. Suppose that Assumptions \ref{assump-A_matrix-eliptic}-\ref{assump-V-g_bounded} hold. 	
	Then, there exists \(u\in H^1(\Tt^d)\cap L^\infty(\Tt^d)\) solving  \eqref{eq-Ell-bound}.
\end{teo}
\begin{proof}
	By  Proposition \ref{pro-ex-H1-ell_bounded}, it follows that there exists  \(u\in H^1(\Tt^d)\)  solving \eqref{eq-Ell-bound}.  
	We prove \(u\in H^1(\Tt^d)\cap L^\infty(\Tt^d)\).
	
	Let $k_0=\left( \frac{C_H+C_V}{C_g}+1\right) $.
	% and let $\operatorname{Lip}(g)$ be the 
	%Lipschitz constant of $g$.
	For all $v\in H^1(\Tt^d)$, we have 
	\begin{equation}\label{def-weak-sol-ell-bounded}
		\int_{\Tt^d}Dv A Du^T +(H(x,u,Du)+V_\Hh(x)-g(-u))v\dx=0.
	\end{equation}
	To understand the behaviour of \(u\) when its absolute value is large, for  $k>0$,
	we introduce
	\begin{equation}\label{def-G_k}
		G_k(s)=\begin{cases}
			s-k,            & s>k            \\
			\hfill	0, \hfill & -k\leq s\leq k \\
			s+k  ,          & s<-k
		\end{cases}
	\end{equation}
	and
	\begin{equation}\label{def-A_k}
		A_k=\{x\in \Tt^d: |u(x)|>k\}.
	\end{equation}
	Setting  $w=G_k(u)$,  we have \(w\in H^1(\Tt^d)\). 
	Note that
	\begin{eqnarray}\label{proof-def-test_f-teo-ex-Linf-ell-bound}
		\begin{split}
			&w=G_k(u(x))=\chi_{A_k}(|u|-k)\sign(u)=(|u|-k)^+\sign(u),\\
			&Dw=DG_k(u)=\chi_{A_k}Du ,
		\end{split}
	\end{eqnarray}
	where \(\chi_{A_k}\) is the indicator function of  \(A_k\).
	
	Next, we take \(w\) as a test function in \eqref{def-weak-sol-ell-bounded} 
	and get
	\begin{equation}\label{proof-eq1-teo-ex-Linf-ell-bound}
		\int_{\Tt^d}DG_k(u)A Du^T -g(-u)G_k(u)\dx
		=
		\int_{\Tt^d} (-H(x,u,Du)-V_\Hh(x))G_k(u)\dx.
	\end{equation}
	Notice that Assumption \ref{assump-V-g_bounded}
	yields 
	\begin{equation}\label{proof-eq2-teo-ex-Linf-ell-bound}
		\begin{split}
			-g(-u)G_k(u)&=-g(-u)\sign(u)\chi_{A_k}(|u|-k)\\
			&\geq C_g|u| \chi_{A_k}(|u|-k)-\frac{1}{C_g}\chi_{A_k}(|u|-k)\\&\geq (k C_g	-\frac{1}{C_g})|G_k(u)|.
		\end{split}
	\end{equation}
	By combining the preceding inequality with the ellipticity condition,  \eqref{proof-def-test_f-teo-ex-Linf-ell-bound} and  \eqref{proof-eq2-teo-ex-Linf-ell-bound} in \eqref{proof-eq1-teo-ex-Linf-ell-bound}, we obtain
	\[
	\theta	\int_{\Tt^d} \chi_{A_k} DuA Du^T\dx +k C_g\int_{\Tt^d} |G_k(u)|\dx\leq \left(C_H+C_V+\frac{1}{C_g}\right)\int_{\Tt^d}|G_k(u)|\dx,
	\]
	for any $k>0$. Hence,
	\[
	\left(k C_g-\left(C_H+C_V+\frac{1}{C_g}\right)\right)\int_{\Tt^d}|G_k(u)|\dx\leq 0. 
	\]
	Now, 	taking 
	\(k=k_0=\frac{C_H+C_V+\frac{1}{C_g}}{C_g}+1\), 
	we deduce  that
	\(G_{k_0}(u)=0\). 
	This with \eqref{def-A_k} and \eqref{proof-def-test_f-teo-ex-Linf-ell-bound} imply \(|u|\leq k_0\). 
\end{proof}

Next, we prove that in Theorem \ref{teo-ex-Linf-ell-bound}, we can  
remove condition \eqref{def-g-bouned}.

\begin{corollary}\label{cor-teo-ex-Linf-ell-bound}
	Consider the setting of Problem \ref{problem-1}.
	Let \(H\) be such that \eqref{def-H-bouned} holds. Suppose that Assumptions \ref{assump-A_matrix-eliptic}-\ref{assump-V-g_bounded} hold.  Then, there exists \(u\in H^1(\Tt^d)\cap L^\infty(\Tt^d)\) solving  \eqref{eq-Ell-bound}. 
\end{corollary}

\begin{proof}	Let $k_0=\left( \frac{C_H+C_V+\frac{1}{C_g}}{C_g}+1\right) $ and let
	$\bar g(s)=g(s)$ for $|s|<k_0$ and extend it linearly with $\bar g'(s)=C_g$ for $|s|>k_0$. 
	Clearly $\bar g$ satisfies \eqref{def-g-bouned}. 
	Moreover, it satisfies the conditions of Assumption \ref{assump-V-g_bounded}. Hence, considering  the following equation
	\begin{equation}\label{eq-Ell-bound-bg}
		-\div(A Du^T) +H(x,Du)+V_\Hh(x)-\bar{g}(-u)=0,\quad x\in\Tt^d,
	\end{equation}
	we note that it satisfies the conditions of Theorem \ref{teo-ex-Linf-ell-bound}. Therefore, there exists $\bar{u}\in H^1(\Tt^d)\cap L^\infty(\Tt^d)$ solving \eqref{eq-Ell-bound-bg}. Moreover, by the proof of Theorem \ref{teo-ex-Linf-ell-bound}, it follows that $|\bar{u}|\leq k_0$.  Consequently, $\bar{g}(\bar{u})=g(\bar{u})$ and equations \eqref{eq-Ell-bound}, \eqref{eq-Ell-bound-bg} are coincide for $\bar{u}$. Thus, $u=\bar{u} \in H^1(\Tt^d)\cap L^\infty(\Tt^d)$ solves \eqref{eq-Ell-bound}. 
\end{proof}

\subsection{Bounded generalized solution to elliptic equation with quadratic growth}
\label{equad}

We focus back to the Hamiltonian with quadratic growth and combine the results from the previous section 
with an approximation argument to show the existence of bounded solutions to \eqref{eq-Ell1}
(see e.g. \cite{Boccardo1992} and \cite{Boccardo1988ExistenceOB}).  

For any $\varepsilon>0$, we set
\begin{equation}\label{def-H_eps}
	H_\varepsilon(x,p)=\frac{p A p^T}{2+\varepsilon |p A p^T|}.
\end{equation}
Notice that
\begin{equation}\label{def-H_eps-prop}
	|H_\varepsilon(x,p)|\leq \frac{1}{\varepsilon}
\end{equation}
and, by the upper bound in Assumption \ref{assump-A_matrix-eliptic}, we have 
\begin{equation}
	\label{pb}
	|H_\varepsilon(x,p)|\leq  \frac{1}{2}p A p^T\leq  \frac{\theta_1}{2}|p|^2.
\end{equation}
Consider the approximation to \eqref{eq-Ell1}
\begin{equation}\label{eq-Ell1-approx}
	-\div(A Du^T_\varepsilon) +H_\varepsilon(x,Du_\varepsilon)+V_{\Hh}(x)-g(-u_\varepsilon)=0.
\end{equation}

We stress that the bound \eqref{def-H_eps-prop} ensures conditions in Corollary \ref{cor-teo-ex-Linf-ell-bound} are satisfied. 
Then, there exists a solution 
\(u_\varepsilon \in H^1(\Tt^d)\cap L^\infty(\Tt^d)\) 
to \eqref{eq-Ell1-approx}.
To obtain the existence of solutions to  \eqref{eq-Ell1}, we consider the limit of $u_\varepsilon$ as $\varepsilon\to 0$. In Theorem \ref{teo-ex-Linf-sol-quad},
we show that the limit exists and it solves \eqref{eq-Ell1}. The main obstacle is the lack of uniformity of the bounds in the preceding section with respect to $\varepsilon$. 
In particular,  the bound in Corollary \ref{cor-teo-ex-Linf-ell-bound}  depends on the constant in  \eqref{def-H-bouned}, which according to \eqref{def-H_eps-prop} is not uniform as $\varepsilon\to 0$. Thus, in the following proposition, we use  \eqref{pb} to establish uniform in $\varepsilon$ bounds for the solutions to \eqref{eq-Ell1-approx}.  

\begin{proposition}\label{pro-Linf_bound-Ell1} 
	Consider the setting of Problem \ref{problem-1}.
	Suppose that Assumptions \ref{assump-A_matrix-eliptic}-\ref{assump-V-g_bounded} hold. 
	Then, 	there exists a constant \(C_\infty\) which does not depend on $\varepsilon$, such that any solution $u_\varepsilon$ to \eqref{eq-Ell1-approx} satisfies	
	\[
	||u_\varepsilon||_{L^\infty(\Tt^d)}\leq C_\infty\,.
	\]
\end{proposition}
\begin{proof}
	Let \(u_\varepsilon\in H^1(\Tt^d)\cap L^\infty(\Tt^d)\) be a bounded solution to \eqref{eq-Ell1-approx}.
	
	For fixed \(k, \lambda\in \Rr^+\),  let $G_k$ and $A_k$ be as in \eqref{def-G_k} and \eqref{def-A_k}, respectively. Let
	\begin{equation}\label{def-phi}
		\phi(s)=\begin{cases}
			e^{\lambda s}-1, \quad   & s\geq 0 \\
			-e^{-\lambda s}+1, \quad & s\leq 0
		\end{cases}
	\end{equation}
	and 
	\[
	\psi(x)=\phi(G_k(u_\varepsilon)).
	\]
	Because \(\psi\in H^1(\Tt^d)\cap L^\infty(\Tt^d)\), 
	we can use  \(\psi\) as a test function in \eqref{eq-Ell1-approx}. 
	Then,
	\begin{equation}\label{proof-eq1-pro-Linf_bound-Ell1}
		\int_{\Tt^d}D\psi ADu_\varepsilon^T-g(-u_\varepsilon) \psi\dx =	\int_{\Tt^d}\left( \frac{1}{2}H_\varepsilon(Du_\varepsilon)-V_{\bar H}(x)\right) \psi\dx.
	\end{equation}
	Note that
	\begin{equation}\label{def-psi-prop}
		\begin{split}
			&	\psi=\phi(|u_\varepsilon|-k)\chi_{A_k}\sign(u_\varepsilon)=\phi((|u_\varepsilon|-k)^+)
			\sign(u_\varepsilon),
			\\
			&	D	\psi=\phi^\prime((|u_\varepsilon|-k)^+)\chi_{A_k}Du_\varepsilon.
		\end{split}
	\end{equation}
	Next, using the previous identities, we estimate the terms in \eqref{proof-eq1-pro-Linf_bound-Ell1}.
	Recalling  that $V_{\Hh}$ is bounded by the formulation of Problem \ref{problem-1},  by  \eqref{def-psi-prop}, we have
	$$
	\begin{cases}
		|\psi|\leq |\phi((|u_\varepsilon|-k)^+)|=\phi((|u_\varepsilon|-k)^+) & \\
		& \\
		|V_{\Hh}\psi|\leq	C_V \phi((|u_\varepsilon|-k)^+).		
	\end{cases}
	$$
	
	Then, by using Assumption \ref{assump-V-g_bounded}, we get
	\begin{equation*}
		\begin{split}
			-g(-u_\varepsilon)\psi&=-g(-u_\varepsilon)\sign(u_\varepsilon)\phi((|u_\varepsilon|-k)^+)\\&\geq C_g |u_\varepsilon|\phi((|u_\varepsilon|-k)^+)-\frac{1}{C_g} \phi((|u_\varepsilon|-k)^+)
			\\&\geq k C_g \phi((|u_\varepsilon|-k)^+)-\frac{1}{C_g}\phi((|u_\varepsilon|-k)^+).
		\end{split}
	\end{equation*}
	Using the preceding inequalities, \eqref{pb},  \eqref{def-psi-prop} and Assumption  \ref{assump-A_matrix-eliptic} in \eqref{proof-eq1-pro-Linf_bound-Ell1},
	we deduce
	\begin{align}\label{proof-eq2-pro-Linf_bound-Ell1}\notag
		&\theta_0\int_{\Tt^d}|Du_\varepsilon|^2\phi^\prime((|u_\varepsilon|-k)^+)\chi_{A_k}\dx+ k C_g\int_{\Tt^d}\phi((|u_\varepsilon|-k)^+) \dx\\
		&\qquad \leq	\frac{\theta_1}{2}\int_{\Tt^d}|Du_\varepsilon|^2\phi((|u_\varepsilon|-k)^+)\dx+	\left(C_V+\frac 1 {C_g}\right)\int_{\Tt^d}\phi((|u_\varepsilon|-k)^+)\dx.
	\end{align}
	Taking \(\lambda=\frac{\theta_1}{2\theta_0}\) in \eqref{def-phi}, 
	we have
	\[
	\theta_0\phi^\prime((|u_\varepsilon|-k)^+)\chi_{A_k}-\frac{\theta_1}{2}\phi((|u_\varepsilon|-k)^+)\geq\frac{\theta_1}{2}\chi_{A_k},
	\]
	where we used the inequality $\phi'(s)\geq \lambda \phi(s)$ for $s\geq 0$. 
	This and \eqref{proof-eq2-pro-Linf_bound-Ell1}, yield
	\[
	\frac{\theta_1}{2}\int_{\Tt^d}|Du_\varepsilon|^2\chi_{A_k}\dx+ k C_g\int_{\Tt^d}\phi((|u_\varepsilon|-k)^+) \dx \leq	\left(C_V+\frac 1 {C_g}\right)\int_{\Tt^d}\phi((|u_\varepsilon|-k)^+)\dx.
	\]
	Finally, taking \(k=k_0=\frac{\left(C_V+\frac 1 {C_g}\right)}{C_g}+1\), we get \(\phi((|u_\varepsilon|-k_0)^+)=0\); that is,  \(|u_\varepsilon|\leq k_0\).
\end{proof}

\begin{teo}\label{teo-ex-Linf-sol-quad}  	Consider the setting of Problem \ref{problem-1}. Suppose that Assumptions \ref{assump-A_matrix-eliptic}-\ref{assump-V-g_bounded} hold. Then,  there exists $u\in H^1(\Tt^d)\cap L^\infty(\Tt^d)$ solving  \eqref{eq-Ell1} in the sense of distributions.
\end{teo}
\begin{proof}	
	Let $u_\varepsilon$ solve \eqref{eq-Ell1-approx}.
	By Proposition \ref{pro-Linf_bound-Ell1},  we know that
	\begin{equation}\label{proof-Linf_bound-teo-ex-Linf-sol-quad}
		\|u_\varepsilon\|_{L^\infty(\Tt^d)}\leq C_\infty,
	\end{equation}
	for some $C_\infty$ which does not depend on $\varepsilon$. 
	
	To show uniform boundedness of  \(\{u_\varepsilon\}\) in \(H^{1}(\Tt^d)\), we set 
	$\Phi(s)=e^{s}$ and notice that
	\(\Phi(u_\varepsilon)\in H^1(\Tt^d)\cap L^\infty(\Tt^d)\).
	By taking \(\Phi(v_\varepsilon)\) as a test function in \eqref{eq-Ell1-approx}, we obtain
	\begin{equation}\label{proof-eq1-teo-ex-Linf-sol-quad}
		\int_{\Tt^d} H_\epsilon(Du_\varepsilon) \Phi(u_\varepsilon)+ \frac{1}{2}Du_\varepsilon ADu_\varepsilon^T \Phi(u_\varepsilon)\dx=	\int_{\Tt^d} (g(-u_\varepsilon)-V_\Hh(x)) \Phi(u_\varepsilon)\dx.
	\end{equation}
	From \eqref{proof-Linf_bound-teo-ex-Linf-sol-quad} it follows that
	\begin{equation}
		\label{mb}
		e^{-C_\infty}\leq \Phi(u_\varepsilon)\leq e^{C_\infty}.
	\end{equation}
	Using this bound together with Assumption \ref{assump-A_matrix-eliptic}
	in  \eqref{proof-eq1-teo-ex-Linf-sol-quad}, yields
	\[
	\int_{\Tt^d}|Du_\varepsilon |^2\dx\leq C, 
	\]
	for some constant independent on $\varepsilon$. 
	Consequently, taking into account \eqref{proof-Linf_bound-teo-ex-Linf-sol-quad}, there exists \(u\in H^1(\Tt^d)\cap L^\infty(\Tt^d)\) such that
	\begin{equation}\label{proof-eq2-teo-ex-Linf-sol-quad1}
		\begin{split}
			&u_\varepsilon\rightharpoonup u\quad \text{in}\quad H^1(\Tt^d),\\
			&	u_\varepsilon\stackrel{*}{\rightharpoonup} u\quad \text{in}\quad L^\infty(\Tt^d).
		\end{split}
	\end{equation}
	
		Note that to complete the proof it is enough  to prove that the limit in the weak form of 
	\eqref{eq-Ell1-approx} is exits, as $\varepsilon\to 0$.  For that first, we prove that 
	\begin{equation}\label{proof-eq6-teo-ex-Linf-sol-quad}
		Du_\varepsilon\to Du,\quad\text{strongly in}\quad L^2(\Tt^d).
	\end{equation}
	
	%{\color{blue}	
		
Taking 
		\[
		\Theta=e^{\mu (u_\varepsilon-u)^2}(u_\varepsilon-u),
		\]	
		as a test function in \eqref{eq-Ell1-approx}, we get
		\begin{equation}\label{mainp1}
\begin{split}
			&2\mu \int_{\Tt^d} e^{\mu (u_\varepsilon-u)^2} (u_\varepsilon-u)^2 (Du_\varepsilon-Du)A Du_\varepsilon^T\dx
+\int_{\Tt^d} e^{\mu (u_\varepsilon-u)^2}  (Du_\varepsilon-Du)A Du_\varepsilon^T\dx\\
&=
\int_{\Tt^d} \left(
-H_\varepsilon(Du_\varepsilon)
-V_{\Hh}+g(-u_\varepsilon)
\right)  e^{\mu (u_\varepsilon-u)^2} (u_\varepsilon-u)\dx.
\end{split}
		\end{equation}
		Observe that 
\begin{equation}\label{mainp2}
		\int_{\Tt^d} \left(
	-V_{\Hh}+g(-u_\varepsilon)
	\right)  e^{\mu (u_\varepsilon-u)^2} (u_\varepsilon-u)\dx\to 0,
\end{equation}
		as $\varepsilon\to 0$. Moreover, because \(u\in H^1(\Tt^d)\cap L^\infty(\Tt^d)\) from \eqref{proof-eq2-teo-ex-Linf-sol-quad1}, we have 
		\begin{equation*}
\begin{split}
		e^{\mu (u_\varepsilon-u)^2} (u_\varepsilon-u)^2 Du \to 0, \quad \text{strongly in}\quad L^2(\Tt^d),\\
				e^{\mu (u_\varepsilon-u)^2}  Du\to Du, \quad \text{strongly in}\quad L^2(\Tt^d)
\end{split}
		\end{equation*}
	 as $\varepsilon\to 0$. Consecutively,
		we have
\begin{equation}\label{mainp4}
\begin{split}
			&\limsup_{\varepsilon\to 0}2\mu \int_{\Tt^d}\Theta (u_\varepsilon-u) (Du_\varepsilon-Du)A Du_\varepsilon^T
	+e^{\mu (u_\varepsilon-u)^2}  (Du_\varepsilon-Du)A Du_\varepsilon^T\dx\\	
	&=\limsup_{\varepsilon\to 0}2\mu \int_{\Tt^d} \Theta (u_\varepsilon-u) Du_\varepsilon A Du_\varepsilon^T
	+ e^{\mu (u_\varepsilon-u)^2}  (Du_\varepsilon-Du)A (Du_\varepsilon^T-Du^T)\dx.
\end{split}
\end{equation}
		Furthermore, we notice that
		\begin{equation}\label{imp}
		\begin{split}
	\limsup_{\varepsilon\to 0}&\int_{\Tt^d} e^{\mu (u_\varepsilon-u)^2}  Du_\varepsilon A Du_\varepsilon^T\dx=
\limsup_{\varepsilon\to 0}\Bigg(\int_{\Tt^d} e^{\mu (u_\varepsilon-u)^2}  (Du_\varepsilon-Du)A (Du_\varepsilon-Du)^T\\
&+2  e^{\mu (u_\varepsilon-u)^2}  (Du_\varepsilon-Du)A Du^T+
 e^{\mu (u_\varepsilon-u)^2}  DuA Du^T\dx\Bigg)\\
&=\limsup_{\varepsilon\to 0}\int_{\Tt^d} e^{\mu (u_\varepsilon-u)^2}  (Du_\varepsilon-Du)A (Du_\varepsilon-Du)^T
+
e^{\mu (u_\varepsilon-u)^2}  DuA Du^T\dx.
		\end{split}
		\end{equation}
		The last equation follows from \eqref{proof-eq2-teo-ex-Linf-sol-quad1}.
On the other hand,	using that $A$ is elliptic and  Young's inequality, $\frac{1}{a}+a(u_\varepsilon-u)^2 \geq 2 |u_\varepsilon-u| $,  we obtain
	\begin{equation}\label{mainp3}
		\begin{split}
			&-\int_{\Tt^d} H_\varepsilon(Du_\varepsilon)e^{\mu (u_\varepsilon-u)^2} (u_\varepsilon-u)\dx\leq 
			C\int_{\Tt^d} |Du_\varepsilon|^2 e^{\mu (u_\varepsilon-u)^2} |u_\varepsilon-u|\dx\\
			&\leq 
			\frac C \mu  \int_{\Tt^d} e^{\mu (u_\varepsilon-u)^2}  Du_\varepsilon A Du_\varepsilon^T\dx
			+2\mu \int_{\Tt^d} e^{\mu (u_\varepsilon-u)^2} (u_\varepsilon-u)^2 Du_\varepsilon A Du_\varepsilon^T\dx. 
		\end{split}
	\end{equation}
Using equations in  \eqref{mainp2}, \eqref{mainp4}, \eqref{imp}, and \eqref{mainp3} in \eqref{mainp1}, we deduce that 
	\begin{equation}\label{mainp5}
\begin{split}
		&	\limsup_{\varepsilon\to 0} \int_{\Tt^d} e^{\mu (u_\varepsilon-u)^2}  (Du_\varepsilon-Du)A (Du_\varepsilon^T-Du^T)\dx\\
&\leq 
\limsup_{\varepsilon\to 0}
\frac{C}{\mu} \int_{\Tt^d} e^{\mu (u_\varepsilon-u)^2}  (Du_\varepsilon-Du)A (Du_\varepsilon-Du)^T\dx
+ e^{\mu (u_\varepsilon-u)^2}  DuA Du^T\dx.
\end{split}
	\end{equation}
Because  $e^{\mu (u_\varepsilon-u)^2} \geq 1$,   we have
	\begin{equation}\label{mainp6}
		\begin{split}
			&\limsup_{\varepsilon\to 0}  \int_{\Tt^d} (Du_\varepsilon-Du)A (Du_\varepsilon-Du)^T\dx\\
&
\leq
\limsup_{\varepsilon\to 0} \int_{\Tt^d} e^{\mu (u_\varepsilon-u)^2}  (Du_\varepsilon-Du)A (Du_\varepsilon^T-Du^T)\dx. 
		\end{split}
	\end{equation}
		Notice that by the dominated convergence theorem, we have
		\[
		\int_{\Tt^d} e^{\mu (u_\varepsilon-u)^2}  DuA Du^T\dx\to \int_{\Tt^d}   DuA Du^T\dx.
		\]
	Relying on this and using \eqref{mainp6} in \eqref{mainp5}, for $\mu$ large enough, we get 
		\[
		\limsup_{\varepsilon\to 0}  \frac 1 2 \int_{\Tt^d} (Du_\varepsilon-Du)A (Du_\varepsilon-Du)^T\dx
		\leq \frac{C}{\mu} \int_{\Tt^d} DuA Du^T\dx.
		\]
		Because $\mu$ is arbitrary, this implies \eqref{proof-eq6-teo-ex-Linf-sol-quad}.

		Now, we ready to prove that the limit passing in the weak form of 
		\eqref{eq-Ell1-approx} is allowed; that is, the limit is allowed in the following equation 
		\begin{equation}\label{eq-Ell1-approx-weak}
			\int_{\Tt^d}Dv ADu_\varepsilon^T+ \left( H_\varepsilon(x,Du_\varepsilon)+V_\Hh(x)-g(-u_\varepsilon)\right) v\dx=	0,
		\end{equation}
	for all \(v\in H^1(\Tt^d)\cap L^\infty(\Tt^d)\).
		First, note that for all \(v\in H^1(\Tt^d)\cap L^\infty(\Tt^d)\)
		\begin{equation}\label{proof-eq7-teo-ex-Linf-sol-quad}
			\begin{split}
				\left| 	\int_{\Tt^d}Dv ADu_\varepsilon^T-Dv ADu^T \dx \right| &\leq  \int_{\Tt^d}\left| Dv A(Du_\varepsilon^T-Du^T) \right| \dx\\&\leq 	||Dv |A|||_{L^2(\Tt^d)}	||Du_\varepsilon-Du||_{L^2(\Tt^d)} \to 0,
			\end{split}
		\end{equation}
	as $\varepsilon \to 0$. Let $\delta, M>0$. Because $Du_\epsilon\to Du$ strongly in $L^2$, we have 
		\begin{align*}
			\int_{\Tt^d}\left| H_\varepsilon(x,Du_\varepsilon)-H_\varepsilon(x,Du) \right| \dx&\leq
			\int_{|Du-Du_\epsilon|<\delta \wedge |Du|<M}\left| H_\varepsilon(x,Du_\varepsilon)-H_\varepsilon(x,Du) \right| \dx\\
			&+\int_{|Du-Du_\epsilon|\geq \delta \vee |Du|\geq M}\left| H_\varepsilon(x,Du_\varepsilon)-H_\varepsilon(x,Du) \right| \dx.
		\end{align*}
		Because $Du$ and $Du_\epsilon$ take values on a compact set for $|Du-Du_\epsilon|<\delta \wedge |Du|<M$, 
		we have
		\[
		\int_{|Du-Du_\epsilon|<\delta \wedge |Du|<M}\left| H_\varepsilon(x,Du_\varepsilon)-H_\varepsilon(x,Du) \right| \dx\to 0, 
		\]
		as $\varepsilon\to 0$. 
		Moreover, 
		\begin{align*}
			&\int_{|Du-Du_\epsilon|\geq \delta \vee ||Du|\geq M}\left| H_\varepsilon(x,Du_\varepsilon)-H_\varepsilon(x,Du) \right| \dx\\&\quad \leq 
			C\int_{|Du-Du_\epsilon|\geq \delta \vee |Du|\geq M} |Du_\varepsilon|^2+|Du|^2 \dx.
		\end{align*}
		Therefore, 
		\begin{align}
			\label{conv1}
			\notag
			&\limsup_{\epsi\to 0}\int_{\Tt^d}\left| H_\varepsilon(x,Du_\varepsilon)-H_\varepsilon(x,Du) \right| \dx\\\notag
			&\leq C \limsup_{\epsi\to 0}\int_{|Du-Du_\epsilon|\geq \delta \vee |Du|\geq M} |Du_\varepsilon|^2+|Du|^2 \dx.\\
			&\leq C \limsup_{\epsi\to 0}\int_{|Du-Du_\epsilon|\geq \delta \vee |Du|\geq M} |Du_\varepsilon-Du|^2+2|Du|^2 \dx.
		\end{align}
	Observe that  the dominated convergence theorem implies
		\[
		\limsup_{\epsi\to 0}\int_{|Du-Du_\epsilon|>\delta \vee |Du|>M} |Du|^2 \dx=
		\int_{|Du|>M} |Du|^2\dx.
		\]
		Furthermore, this with again dominated convergence theorem, yield
		\begin{align}
			\label{conv1A}	
			\notag
			&\limsup_{\delta\to 0, M\to \infty}\limsup_{\epsi\to 0}\int_{|Du-Du_\epsilon|>\delta \vee |Du|>M} |Du|^2 \dx\\
			&=\limsup_{M\to \infty}\int_{|Du|>M} |Du|^2\dx=0.
		\end{align}
		On the other hand, 
		\begin{align}
			\label{conv1B}
			\int_{|Du-Du_\epsilon|>\delta \vee |Du|>M} |Du_\varepsilon-Du|^2\dx
			\leq \int_{\Tt^d}|Du_\varepsilon-Du|^2\dx\to 0. 
		\end{align}
		Consequently, combining \eqref{conv1A} with \eqref{conv1B} in \eqref{conv1}, we conclude that 
		\[
		\int_{\Tt^d}\left| H_\varepsilon(x,Du_\varepsilon)-H_\varepsilon(x,Du) \right| \dx\to 0, 
		\]
		as $\varepsilon\to 0$. 
		
From the definition of $H_\varepsilon$ (see \eqref{def-H_eps}), we have
		\[
		\left| H_\varepsilon(x,Du )-\frac 1 2  Du A Du^T\right|\leq \frac 1 2  Du A Du^T, 
		\]
		which with the dominated convergence theorem  imply
		\[
		\int_{\Tt^d}\left| H_\varepsilon(x,Du )-\frac 1 2  Du A Du^T\right|\dx \to 0, 
		\]
		as $\varepsilon\to 0$. 
		Thus, we have
		\begin{equation}\label{proof-eq8-teo-ex-Linf-sol-quad}
			\begin{split}
				\int_{\Tt^d}	\left| H_\varepsilon(x,Du_\varepsilon)-\frac{1}{2}Du ADu^T\right|  \dx  &\leq  \int_{\Tt^d}\left| H_\varepsilon(x,Du_\varepsilon)-H_\varepsilon(x,Du) \right| \dx\\&+ \int_{\Tt^d}\left| \frac{1}{2}Du ADu^T-H_\varepsilon(x,Du)  \right| \dx\to 0. 
			\end{split}
		\end{equation}

		Finally, Rellich-Kondrachov Theorem and  \eqref{proof-eq6-teo-ex-Linf-sol-quad}  imply that \(u_\varepsilon\to u\) strongly in  \(H^1(\Tt^d)\). Therefore, because
		$u_\varepsilon$ and $u$ are bounded in $L^\infty$ and 
		$g$ is locally Lipschitz, we get
		\begin{equation}\label{proof-eq9-teo-ex-Linf-sol-quad}
			\int_{\Tt^d}	\left| g(-u_\varepsilon)-g(-u)\right|  \dx  \leq C  \int_{\Tt^d}\left| u_\varepsilon -u \right| \dx\to 0.
		\end{equation}
		Using  \eqref{proof-eq7-teo-ex-Linf-sol-quad},  \eqref{proof-eq8-teo-ex-Linf-sol-quad} and \eqref{proof-eq9-teo-ex-Linf-sol-quad} in \eqref{eq-Ell1-approx-weak} and letting \(\varepsilon\to0\), we conclude that 
		\begin{equation*}%\label{eq-Ell1-approx-weak2}
			\int_{\Tt^d}Dv ADu ^T+ \left(\frac 1 2 Du A Du^T+V_\Hh(x)-g(-u)\right) v\dx=	0.\qedhere
		\end{equation*}
	\end{proof}

	\begin{proposition}\label{pro-C^a-reg}
		Suppose that Assumptions \ref{assump-A_matrix-eliptic}-\ref{assump-V-g_bounded} hold and let \(u\in H^1(\Tt^d)\cap L^\infty(\Tt^d)\) be a solution to  \eqref{eq-Ell1}. 
		Then, \(u\in H^1(\Tt^d)\cap C^{0,\gamma}(\Tt^d)\) for some \(\gamma>0\).
	\end{proposition}
	\begin{proof}  Because $u$ is a weak solution to \eqref{eq-Ell1} it satisfies 
		\begin{equation}\label{pro-gen-Cacci-weak0}
			\int_{\Tt^d} DuA(x) D\varphi^T\dx=-\frac{1}{2}	\int_{\Tt^d}DuA(x)Du^T\varphi\dx+\int_{\Tt^d} (g(-u)-V_{\Hh}) \varphi\dx,
		\end{equation} 
		for all \(\varphi\in  H^1(\Tt^d)\cap L^\infty(\Tt^d)\). Taking into account that \(u\in H^1(\Tt^d)\cap L^\infty(\Tt^d)\), we  set $\varphi=e^{-\frac{u}{2}}\psi$ with \(\psi\in  H^1(\Tt^d)\cap L^\infty(\Tt^d)\)  as a test function in \eqref{pro-gen-Cacci-weak0} to get 
		\begin{equation*}
			\begin{split}
				\int_{\Tt^d} \Big(-\tfrac{1}{2}e^{-\frac{u}{2}}\psi Du+e^{-\frac{u}{2}}D\psi\Big)A(x)Du^T +\Big(\frac{1}{2}	DuA(x)Du^T- g(-u)+V_{\Hh}\Big )e^{-\frac{u}{2}}\psi\dx\\=	\int_{\Tt^d} D\psi e^{-\frac{u}{2}}A(x)Du^T -\Big(g(-u)-V_{\Hh}\Big )e^{-\frac{u}{2}}\psi\dx=0.
			\end{split}
		\end{equation*} 
		Because \(u\in  L^\infty(\Tt^d)\), the preceding equation can be written as follows
		\begin{equation}\label{Ho-main1}
			\begin{split}
				\int_{\Tt^d} D\psi A_e(x)Du^T -f_e\psi\dx=0,
			\end{split}
		\end{equation} 
		where $A_e(x)=e^{-\frac{u}{2}}A(x)\in (L^\infty(\Tt^d))^{2d}$ and $f_e=\Big(g(-u)-V_{\Hh}\Big )e^{-\frac{u}{2}}\in L^\infty(\Tt^d)$. 
		Let  \(u^*\) be the periodic extension of \(u\) to \(\Rr^d\). 
		Then, from \eqref{Ho-main1}, we have
		\begin{equation}\label{Ho-eq-main1}
			\int_{\Rr^d} D\psi A_e(x)(Du^*)^T -f_e\psi\dx=0,
		\end{equation}
		for all \(\psi\in  H_0^1(\Rr^d)\). Note that  if  $u^*\in  C^{0,\alpha}(\Omega_2)$ for some $\Omega_2\subset\Rr^d$ satisfying  \(\Omega_1= [-1,1]^d\subset\Omega_2\), then,  $u\in C^{0,\alpha}(\Tt^d)$. So, to prove H{\"o}lder regularity of $u^*$ in a domain containing $\Omega_1$, 
		we fix smooth domains \(\Omega^\prime_ 1 \) and $\Omega_2$ such that 
		\(\Omega_1\subset\Omega^\prime_ 1\subset\Omega_ 2\).  Using these sets, we define 
		\begin{equation}\label{def-zeta-2}
			\eta_2(x)=\begin{cases}
				1&\quad x\in \Omega_1\\ 
				\zeta_2(x)&\quad x\in \Omega_2\setminus \Omega_1
				\\
				0,&\quad 0\in \Rr^d\setminus\Omega_2,
			\end{cases}
		\end{equation}
		where \(\zeta_2\in C_0^{2,\alpha}(\Rr^d)\) is such that $\zeta_2(x)>\frac{1}{2}$ for all $x\in \Omega^\prime_1$, \(\eta_2
		\in C_0^{2,\alpha}(\Rr^d)\), \(|D\eta_2|\leq C\).  
		Taking $\psi(x)=\varphi(x)\eta_2(x)$, $\varphi\in H^1(\Rr^d)$ as a test function  in \eqref{Ho-eq-main1}, we get
		\begin{equation}\label{Ho-eq-main2}
			\begin{split}
				\int_{\Rr^d}D\varphi A_e (D(\eta_2 u^*))^T+\varphi D\eta_2 A_e (D u^*)^T-D\varphi u^*A_e (D\eta_2 )^T- f_e\eta_2\varphi\, \dx=0.
			\end{split}
		\end{equation}
		Note that 
		\begin{equation}\label{Ho-eq-main3}
			\begin{split}
				\int_{\Rr^d}\varphi D\eta_2 A_e (Du^*)^T\, \dx=-\int_{\Rr^d}\div(\varphi D\eta_2 A_e )u^*\, \dx\\=-\int_{\Rr^d}u^*D\varphi( D\eta_2 A_e )^T+\varphi u^*\div(D\eta_2 A_e ) \, \dx.
			\end{split}
		\end{equation}
		Substituting \eqref{Ho-eq-main3} into \eqref{Ho-eq-main2}, we get
		\begin{equation}\label{Ho-eq-main4}
			\begin{split}
				\int_{\Rr^d}D\varphi A_e D(\eta_2 u^*)^T-D\varphi (2u^*A_e (D \eta_2)^T)- (f_e\eta_2+u^*\div(D\eta_2 A_e ))\varphi\, \dx=0.
			\end{split}
		\end{equation}
		Denoting $u_\eta=\eta_2u^*$, we notice that $u_\eta$ solves the following boundary value problem
		\begin{equation}\label{Ho-eq-boundary}
			\begin{cases}
				-\div(A_e Du_\eta^T)+\div(f_b) -f_h =0	&\quad \text{in}\quad\Omega_2	\\
				u_\eta=0,	&	\quad \text{on}\quad\partial\Omega_2
			\end{cases}
		\end{equation}
		where 
		\begin{equation*}
			\begin{split}
				f_b=2u^*A_e(D\eta_2 )^T,\quad f_h= f_e\eta_2+u^*\div(D\eta_2 A_e ).
			\end{split}
		\end{equation*}
		Recalling the definitions of $\eta_2$ and $A_e$, and that  $u^*$ is bounded on $\Omega_2$, we deduce that $A_e$ is uniformly elliptic and $A_e\in (L^\infty(\Omega_2))^{2d}$, $f_b\in (L^\infty(\Omega_2))^{d}$, $f_h\in L^\infty(\Omega_2)$. Therefore, because $u^*\in H^1(\Omega)\cap L^\infty(\Omega)$ and solves elliptic equation in \eqref{Ho-eq-boundary} by the classic regularity results for linear elliptic equations (see, for example, Theorem 1.1 in Chapter 4 of  \cite{Lad}), we obtain that $u_\eta\in C^{0,\gamma}(\Omega_2)$ for some $\gamma>0$.  Recalling that $\eta_2>\frac{1}{2}$ on $\Omega_1$, we deduce  that  $u^*\in C^{0,\gamma}(\Omega_1)$, hence,  $u_\eta\in C^{0,\gamma}(\Tt^d)$.
	\end{proof}

\section{Regularity Analysis}
\label{sec:5}

In this section, we investigate the regularity of linear elliptic boundary value problems
and prove H{\"o}lder continuity of the corresponding solutions. 
Here, $\Omega$ is a bounded domain with a smooth boundary in $\Rr^d$.

We recall the definitions of Morrey and Campanato spaces. 

\begin{definition}[Morrey spaces]
	Let $\Omega \subset \mathbb{R}^d$ be an open set with smooth boundary and let $B(x,r)$ be any ball centered at $x$ with radius $r$. 
	We say that a locally integrable function $f$ belongs to the Morrey space $L^{p,\lambda}(\Omega)$,  $1 \leq p < \infty$, $0<\lambda <d$ if
	$$
	\|f\|_{L^{p,\lambda}(\Omega)} = \sup_{x \in \Omega, r > 0} \left(\frac{1}{r^{\lambda}} \int_{B(x,r)\cap \Omega} |f(y)|^p dy\right)^{\frac{1}{p}} < \infty.
	$$
	
	We say that a locally integrable function $f$ belongs to the weak Morrey space $L_w^{p,\lambda}(\Omega)$ if there exists $c\geq 0$ such that
	$$
	\sup_{t>0}t^{p}|\{x\in \Omega:|f(x)|>t\}\cap B(x,r)|\leq c \, r^\lambda.
	$$
	The best constant in the previous inequality will be denoted by $\|f\|_{L_w^{p,\lambda}(\Omega)}.$
	
\end{definition}

\begin{definition}[Campanato spaces]
	Let $\Omega \subset \mathbb{R}^d$ be an open set with smooth boundary and let $B(x,r)$ be any ball centered at $x$ with radius $r$. 
	We say that a locally integrable function $f$ belongs to the 
	Campanato space $\mathcal{L}^{p,\lambda}(\Omega)$, $1 \leq p < \infty,$ $\lambda > 0$ if 
	$$
	\|f\|_{\mathcal{L}^{p,\lambda}(\Omega)} 
	=
	\sup_{x \in \Omega, r > 0} \left(\frac{1}{r^d} \int_{B(x,r) \cap \Omega} |f(y) - f_{B(x,r)}|^p dy\right)^{\frac{1}{p}} < \infty
	$$
	where $f_{B(x,r)}=\frac{1}{|B(x,r)|}\int_{B(x,r)}f(x)\dx$ denotes the integral average of $f$ on $B(x,r)$.
\end{definition}

Campanato and Morrey spaces are closely related and valuable tools for analysing partial differential equations.

Here, we prove H{\"o}lder continuity for solutions to the following problem.
\begin{problem}
	\label{p2}
	Let $0<\alpha<1$ and $\lambda\in(d-2,d)$.
	Suppose that $f_1\in C^{0,\alpha}(\Omega)$, $f_2,f_3\in (L^{2,\lambda}(\Omega))^d$ and $f_4\in L^{1,\lambda}(\Omega)$. Find $v\in H^1_0(\Omega)$ such that
	\begin{equation}\label{eq-Ell}
		\begin{cases}
			-\div(ADv^T )+f_1 v+\div(f_2)+f_3Dv+f_4=0\quad & {\text{in}}\quad \Omega         \\
			v=0\quad                                       & {\text{on}}\quad \partial\Omega.
		\end{cases}
	\end{equation}
\end{problem}

\subsection{ Weak Solutions}

We start with some results on Morrey spaces  
crucial to Problem \ref{p2}.

\begin{lemma}\label{lem2} 
	Let $d\geq 3$. Suppose that  $u\in H^{1}(\Omega)$ and that $Du\in(L^{2,\nu}(\Omega))^d$ for some $0\leq\nu<d-2$.	
	Then, $u\in L^{2,2+\nu}(\Omega)$.
\end{lemma}
\begin{proof}
	Let $p\geq1$ be such that $\frac{1}{2}=\frac{1}{p}-\frac{1}{d}$.
	By Poincar{\'e} 
	%(see Theorem 4.9 in \cite{evansgariepy2015}) 
	and H\"older inequalities
	\begin{equation*}%\label{lem2-proof-main}
		\int_{B_R} |u-\bar{u}_R|^{2}\dx
		\leq 
		C R^{{d+2}-2d/p}\left( 	\int_{B_R} |Du|^{p}\dx\right) ^{\frac{2}{p}}
		\leq
		CR^{2} \int_{B_R} |Du|^{2}\dx.
	\end{equation*}
%	and then
%	\begin{equation*}
%		\int_{B_R} |u-\bar{u}_R|^{2}\dx\leq  C R^{2}\left( \int_{B_R} |Du|^{2}\dx \right).
%	\end{equation*}
	Because $Du\in(L^{2,\nu}(\Omega))^d$, $u$ belongs to  Campanato space $\mathcal{L}^{2,2+\nu}(\Omega)$ where $2+\nu<d $. Recalling that the for $1\leq k<d$ Campanato space $\mathcal{L}^{2,k}(\Omega)$ coincides with the Morrey space  $L^{2,k}(\Omega)$ (see \cite{Kufner} or \cite{SDH1} and \cite{SDH2}), we conclude  the proof.
\end{proof}

The following Lemma is proved in \cite{DiFazio1993} (see also \cite{SDH1} and \cite{SDH2}).

\begin{lemma}\cite[Lemma 4.1]{DiFazio1993}\label{lem1}
	Let $\lambda\in(d-2,d)$, $\nu\in[0,d-2)$, suppose that $f\in L^{2,\lambda}(\Omega)$, $u\in L^{2,2+\nu}(\Omega)$, and $Du\in L^{2,\nu}(\Omega)$.
	Then, $fu\in L^{2,2+\nu+\lambda-d}(\Omega)$ and
	$$
	||fu||_{L^{2,2+\nu+\lambda-d}}\leq C||f||_{L^{2,\lambda}}(||Du||_{L^{2,\nu}}+||u||_{L^{2,2+\nu}}).
	$$
\end{lemma}

Next, relying on the previous two Lemmas, we prove finiteness of integrals that arise in Definition \ref{Def-weak-for-els}. 

\begin{proposition}
	\label{propfinite}	
	
	Consider the setting of Problem \ref{p2} and Suppose that Assumption  \ref{assump-A_matrix-eliptic} holds.  Let $v$ be a corresponding solution to Problem \ref{p2}. 
	 Then, for all $\varphi\in H_0^1(\Omega)$ the following integral is finite
	\begin{equation}\label{def-weak-sol-to-els2}
		\int_{\Omega}D\varphi(ADv^T-f_2)+(f_1 v+f_3Dv+f_4) \varphi \,dx.
	\end{equation}
\end{proposition}

\begin{proof} 
	Assumption  \ref{assump-A_matrix-eliptic} implies that the first term of \eqref{def-weak-sol-to-els2} is bounded. Therefore, 
	it remains to prove that
	\begin{equation}\label{rem-weak-sol-to-els-eq}
		\int_{\Omega}\left( f_1 v \varphi+f_3Dv \varphi+f_4\varphi \right) \dx
	\end{equation}
	is finite for all $\varphi\in H_0^1(\Omega)$.
	The first term is finite because $f_1\in C^{0,\alpha}(\Omega)$ and  $\varphi,v\in H^1(\Omega)$.
	Moreover, because $\varphi,v\in H^1(\Omega)$, Lemma \ref{lem2} yields $v,\varphi\in L^{2,2}(\Omega)$. 
	Then, taking into account that $f_3\in(L^{2,\lambda}(\Omega))^d$ by Lemma \ref{lem1}, we get  $\varphi f_3 \in (L^{2,4+\lambda-d}(\Omega))^d$,  which implies that $\varphi f_3Dv\in L^1(\Omega)$. 
	Accordingly, the second term in \eqref{rem-weak-sol-to-els-eq} is finite.
	It remains to prove that the third term in \eqref{rem-weak-sol-to-els-eq} is also finite.  So, recalling that
	 $f_4\in L^{1,\lambda}(\Omega)$, we have $|f_4|^\frac{1}{2}\in L^{2,\lambda}(\Omega)$. 
	Hence, using Lemma \ref{lem1}, we obtain $|f_4|^\frac{1}{2}\varphi \in L^2(\Omega)$, which yields
	$|f_4|^\frac{1}{2}|f_4|^\frac{1}{2}\varphi \in L^1(\Omega)$. 
\end{proof}

Proposition \ref{propfinite} ensures the following definition of weak solutions to \eqref{eq-Ell} makes sense as all integrals are finite. 

\begin{definition}\label{Def-weak-for-els}
	We say that $v\in H_0^1(\Omega)$  solves  \eqref{eq-Ell} in the weak sense  if for all $\varphi\in H_0^1(\Omega)$ following equality holds
	\begin{equation*}%\label{def-weak-sol-to-els}
		\int_{\Omega}D\varphi(ADv^T-f_2)+(f_1 v+f_3Dv+f_4) \varphi \dx=0.
	\end{equation*}
\end{definition}

\subsection{Preliminary results} 
We review here some preliminary results regarding equations of the following two forms:
\begin{equation}\label{eq-Ell-v1}
	-\div(ADw^T )=-\div(f_2)
\end{equation}
and
\begin{equation}\label{eq-Ell-v22}
	-\div(ADw^T )=f.
\end{equation}

The first result concerns \eqref{eq-Ell-v1}.
\begin{teo}\cite[Theorem 3.3 ]{DiFazio1993} \label{teo-divf-Holder}
	Let $v_1\in H^1$ solve \eqref{eq-Ell-v1} and suppose that  Assumptions \ref{assump-A_matrix-eliptic} and \ref{assump-A_matrix_element-bounded} hold. Further, assume that  $f_2\in (L^{2,\nu}(\Omega))^d$ with $d-2<\nu<d$. Then, $v_1\in C^{0,\alpha}(\Omega)$ for some $0<\alpha<1$.
\end{teo}

The following theorem is proven in \cite{DiFazio_2020_Ind}.

\begin{teo} \cite[Theorem 4.4]{DiFazio_2020_Ind} \label{teo-f-<d-2}
	Let $w\in H^1(\Omega)$ be a very weak solution to  \eqref{eq-Ell-v22} and $f\in L^{1,\nu}(\Omega)$ with $0<\nu<d-2$. Suppose that Assumptions \ref{assump-A_matrix-eliptic} and \ref{assump-A_matrix_element-bounded} hold. Then,  there exists a positive constant $C$ such that
	$$
	||w||_{L_w^{q_\nu,\nu}(\Omega)}\leq C 	||f||^{q_\nu}_{L^{1,\nu}(\Omega)},
	$$
	where
	$\displaystyle{
		\dfrac{1}{q_\nu}=1- \dfrac{2}{d-\nu}
	}$.
\end{teo}
As a consequence,  we obtain the following result. 

\begin{corollary} \cite[Remark 2]{MariaRagusa1994regularity} \label{cor-teo-f-<d-2}
	Consider the setting of Theorem \ref{teo-f-<d-2}. Then, there exists  a positive constant $C$ such that 
	$$
	||w||_{L^{p,\nu}(\Omega)} \leq C, 	
	$$
	where $1\leq p<q_\nu$ with
	$\displaystyle{
		\dfrac{1}{q_\nu}=1- \dfrac{2}{d-\nu}
	}$. Moreover, 
	$$
	||w||_{L^{1,2+\nu}(\Omega)}
	\leq C.
	$$
\end{corollary}
\begin{proof} The first estimate follows from Theorem  \ref{teo-f-<d-2} and the following embeddings of Morrey spaces
	\begin{equation*}
		L^{p,\lambda}_w(\Omega)\subset L^{q,\lambda}(\Omega), \quad 1\leq q<p. 
	\end{equation*}
	Next, we prove the estimate for $\|w\|_{L^{1,2+\nu}(\Omega)}$. 
Note that
	\begin{equation}\label{cor-teo-f-<d-2-L2,2+nu-eq1}
		\begin{split}
			\int_{B_R\cap\Omega}|w|\dx&=\int_{0}^{+\infty}|\{x\in B_R\cap\Omega:|w|>t\}|\dt\\&=
			\int_{0}^{\tau}|\{x\in B_R\cap\Omega:|w|>t\}|\dt+\int_{\tau}^{+\infty}|\{x\in B_R\cap\Omega:|w|>t\}|\dt,
		\end{split}
	\end{equation}
	for all $\tau>0$. On the other hand, from the definitions of Morrey, weak Morrey spaces, and   Theorem \ref{teo-f-<d-2}, we have
	$$
	t^{q_\nu}|\{x\in B_R\cap\Omega:|w|>t\}|\leq C R^\nu.
	$$
	Consequently, by \eqref{cor-teo-f-<d-2-L2,2+nu-eq1} and the fact, $\nu<d-2$, we deduce 
	\begin{align*}
		\int_{B_R\cap\Omega}|w|\dx&\leq
		C\left( R^d \int_{0}^{\tau}\dt+R^\nu\int_{\tau}^{+\infty}t^{-q_\nu}\dt\right)\\&\leq
		C\left( R^d \tau+
		\dfrac{R^\nu}{1-q_0} \tau^{1-q_\nu}\right).
	\end{align*}
	 Taking $\tau=R^{-(d-\nu-2)}$ in the preceding equation, we complete the proof.
\end{proof}

The following result provides H\"older continuity for solutions to \eqref{eq-Ell-v22}. 

\begin{teo}\cite[Theorem 4.13 ]{DiFazio_2020_Ind}\label{teo-f-Holder} Let $w$ be a solution to \eqref{eq-Ell-v22} and suppose that Assumptions \ref{assump-A_matrix-eliptic} and \ref{assump-A_matrix_element-bounded} hold. Further, assume $f\in L^{1,\nu}(\Omega)$ where $d-2<\nu<d$. Then, $w\in C^{0,\alpha}(\Omega)$ for some $0<\alpha<1$.
\end{teo}

\subsection{Regularity of  Weak Solutions} 
In this section,  we obtain H{\"o}lder continuity of weak solutions to Problem \ref{p2}.
More precisely, if $v$ solves \eqref{eq-Ell}, then
$v\in C^{0,\alpha}(\Omega)$ for some $\alpha\in(0,1)$. 
This achieved by combining results from previous sections with the following Lemmas on elliptic PDEs in Morrey spaces (see, e.g., \cite{DiFazio_2020_Ind}, \cite{SDH1} and \cite{SDH2}). 

We start by  splitting the equation in \eqref{eq-Ell} into two problems and analysis their solutions separately.
\begin{remark}\label{pro-v_2}
	Consider the setting of Problem \ref{p2}.
	Let $v\in H^1(\Omega)$ solves \eqref{eq-Ell} and let $v_1\in H^1(\Omega)$ solve \eqref{eq-Ell-v1}.
	Then, $v_2=v-v_1$ solves
	\begin{equation}\label{eq-Ell-v2}
		-\div(ADv_2^T )=-f_1 v-f_3Dv-f_4.
	\end{equation}
\end{remark}

The following Lemma  provides an interpolation result for Morrey spaces.  
\begin{lemma}  \cite[Lemma 3]{MariaRagusa1994regularity} \label{lem3}Let $0\leq\nu,\mu<d$, $u\in L^{\frac{2d}{d-2},\nu}(\Omega)\cap L^{1,\mu}(\Omega)$. Then, $u\in L^{2,\theta}(\Omega)$, where $\theta=\frac{4\mu+\nu(d-2)}{d+2}$.
\end{lemma}

By applying  H{\"o}lder inequality, we get next lemma, which provides  estimates for the product of functions in Morrey spaces. 

\begin{lemma}\label{lem-est-f3dv} Suppose that $u\in L^{2,\nu}(\Omega)$ and  $f\in  L^{2,\mu}(\Omega)$.  Then, $fu \in L^{1,\frac{\nu+\mu}{2}}(\Omega)$ 
	and
	$$
	\|fu\|_{L^{1,\frac{\nu+\mu}{2}}} \leq 
	\|f\|^{\frac{1}{2}}_{L^{2,\mu}} \|u\|^{\frac{1}{2}}_{L^{2,\nu}}.
	$$
\end{lemma}

The following result is  Morrey space version of Sobolev Theorem.

\begin{lemma}\label{lem-est-Sobolev} Suppose that $u\in H^{1}(\Omega)$ and  $Du\in  L^{2,\nu}(\Omega)$ for  $0\leq\nu< d-2$.  Then, $u \in L^{\frac{2d}{d-2},\nu\frac{d}{d-2}}(\Omega)$.
\end{lemma}
\begin{proof}
	
	By Poincar{\'e} inequality (see Theorem 4.9 in \cite{evansgariepy2015}) for $2^*=\frac{2d}{d-2}$, we have
	\begin{align*}
		\left( \dfrac{1}{|B_R|}\int_{B_R} |u-\bar{u}_R|^{2^*}\dx\right)^{\frac{1}{2^*}}
		&
		\leq 
		C R 
		\left(\dfrac{1}{|B_R|}\int_{B_R} |Du|^{2}\dx\right)^{\frac{1}{2}}
		\\
		&
		\leq C 	R^{\frac{2-d+\nu}{2}} \|Du\|_{L^{2,\nu}}.
	\end{align*}
Hence,
	\begin{align*}
		\int_{B_R} |u-\bar{u}_R|^{2^*}\dx
		\leq  
		C R^{\nu\frac{ 2^*}{2}} \|Du\|^{2^*}_{L^{2,\nu}}.
	\end{align*}
	Consequently, 
	$$
	\|u\|_{\mathcal{L}^{2^*,\nu\frac{ 2^*}{2}}} \leq C \|Du\|_{L^{2,\nu}}.
	$$
	This implies  that $u\in \mathcal{L}^{\frac{2d}{d-2},\nu\frac{d}{d-2}}(\Omega)$. 
	Recalling that $\nu\frac{d}{d-2}<d $, we conclude the proof by using well-known relations between Morrey and Campanato spaces (see  Theorem 4.3 in  \cite{Kufner}).
\end{proof}

We achieve bound  the norm of the solution gradient  in terms of the norm of the solution itself  by using inequalities in \eqref{lem-est-Cacciopp-state-eq1} and \eqref{lem-est-Cacciopp-state-eq2}.

\begin{lemma}\label{lem-est-Cacciopp} Consider the setting of Problem \ref{p2} and
	suppose that  Assumption  \ref{assump-A_matrix-eliptic}  holds.  Let $v_1$ and $v_2$ be  solutions to \eqref{eq-Ell-v1} and \eqref{eq-Ell-v2}, respectively. 
	Then, there exists a constant $C$ such that  for all  $R>0$
	\begin{equation}\label{lem-est-Cacciopp-state-eq1}
		\int_{B_R}|Dv_1|^2\dx
		\leq 
		\frac{C}{R^2}  \int_{B_{2R}} |v_1-	\bar{v}_1|^2\dx+
		C
		\int_{B_{2R}} |f_2|^2\dx,
	\end{equation}
	and
	\begin{align}\label{lem-est-Cacciopp-state-eq2}
		\int_{B_R}|Dv_2|^2\dx
		& 
		\leq
		C \int_{B_R}|Dv_1|^2\dx
		+
		\frac{C}{R^2}\int_{B_{2R}} (|v_1|^2+|v_2|^2)\dx
		\\
		\nonumber
		&
		+
		C \int_{B_{2R}} (|f_3|^2+|f_4|)(|v_1|^2+|v_2|^2)\dx.
	\end{align}
	%	where
	%	\begin{equation}\label{def-v_1-bar}
		%		\bar{v}_1= \frac{1}{|B_{2R}|}\int_{B_{2R}}v_1~\dx.
		%	\end{equation}	
\end{lemma}
\begin{proof}
	By the definition of weak solution to  \eqref{eq-Ell-v1}, we have
	\begin{equation}\label{lem_Cacci-eq01}
		\int_{\Omega} Dv_1A(x) D\psi^T\dx=\int_{\Omega} f_2 D\psi \dx,
	\end{equation}
for all $ \psi\in H^1(\Omega).$ 	Let $\eta:\Rr^d\to\Rr^+_0$ be a smooth $B_R-B_{2R}$ cut-off function whose gradient is controlled by $C/R$, where $C$ is a universal constant.
	Setting 
	\begin{equation*}%\label{lem_Cacci-def_varphi}
		\psi_1(x)=\eta^2(x)(v_1(x)-\bar{v}_1)
	\end{equation*}
	and taking $\psi=\psi_1$ in \eqref{lem_Cacci-eq01}, 
	we obtain
	\begin{align*}%\label{lem_Cacci-eq02}
		\int_{B_{2R}} \eta^2 Dv_1A(x)Dv_1^T \dx
		&
		=
		-2\int_{B_{2R}} \eta(v_1(x)-\bar{v}_1)Dv_1 A(x)D\eta^T \dx
		\\
		\nonumber
		&
		+
		\int_{B_{2R}} \eta f_2\eta  Dv_1\dx
		+
		2\int_{B_{2R}}(v_1(x)-\bar{v}_1) \eta f_2D\eta  \dx.
	\end{align*}
	By Assumption \ref{assump-A_matrix-eliptic}, we have
	\begin{align*}
		\theta_0\int_{B_{2R}} \eta^2|Dv_1|^2 \,dx
		&
		\leq
		\frac{\theta_0}{2}\int_{B_{2R}} \eta^2|Dv_1|^2  \dx +\frac{C}{R^2} \int_{B_{2R}} \eta^2|v_1-\bar{v}_1|^2\dx
		\\
		&
		+C\int_{B_{2R}}\eta^2| f_2|^2\,dx,
	\end{align*}
	which implies \eqref{lem-est-Cacciopp-state-eq1}. 
	
	Next, we prove \eqref{lem-est-Cacciopp-state-eq2}.	Because $v_2$ solves \eqref{eq-Ell-v2}, we have
	\begin{equation}\label{lem_Cacci-eq1}
		\int_{\Omega} Dv_2A(x) D\psi^T\dx=-\int_{\Omega} (f_1 v+f_3Dv+f_4) \psi\dx, 
	\end{equation}
for all $\psi\in H^1(\Omega)$.
Let
	\begin{equation}\label{lem_Cacci-def_phi}
		\psi(x)=\eta^2(x)v(x)=\eta^2(x)(v_1(x)+v_2(x)),
	\end{equation}
	where  $\eta$ is a smooth cut-off function as before.
	By \eqref{lem_Cacci-eq1} and \eqref{lem_Cacci-def_phi}, we have
	\begin{align*}%\label{lem_Cacci-eq2}
		\int_{B_{2R}} \eta^2 Dv_2 A(x)Dv_2^T \,dx
		&
		=
		-2\int_{B_{2R}} \eta (v_2+v_1) Dv_2A(x)D\eta^T \dx
		\\
		\nonumber
		&
		-\int_{B_{2R}} \eta^2 Dv_2 A(x)Dv_1^T \dx
		\\
		\nonumber
		&
		-\int_{B_{2R}} f_1\eta^2(v_1+v_2)^2\dx
		\\
		\nonumber
		&
		-
		\int_{B_{2R}} \eta^2 f_3 (v_1+v_2) (Dv_1+Dv_2)\dx
		\\
		\nonumber
		&
		+
		\int_{B_{2R}}
		\eta^2 f_4 (v_1+v_2)\dx.
	\end{align*}
	Using Assumption  \ref{assump-A_matrix-eliptic}, we obtain
	\begin{align*}
		\theta\int_{B_{2R}} \eta^2|v_2|^2  \dx
		&
		\leq
		\frac{\theta}{2}\int_{B_{2R}} \eta^2|Dv_2|^2  \dx +C\int_{B_{2R}} \eta^2|Dv_1|^2  \dx
		\\
		&
		+
		\frac{C}{R^2}
		\int_{B_{2R}} \eta^2(|v_1|^2+|v_2|^2)\dx
		\\
		&
		+
		C \int_{B_{2R}} (|f_3|^2+|f_4|)\eta^2(|v_1|^2+|v_2|^2)\dx.
	\end{align*}
This completes the proof.
\end{proof}

%A very useful inequality to prove H\"older continuity is given by the following result.

Now, we ready to prove H{\"o}lder continuity for weak solutions to \eqref{eq-Ell}. 
%by writing $v$ as $v=v_1+v_2$, and proving H{\"o}lder continuity for $v_1$ and $v_2$ individually.
\begin{teo}\label{teo-main-Holder}
Let	$v$ solves Problem \ref{p2} and suppose that   Assumptions \ref{assump-A_matrix-eliptic} and  \ref{assump-A_matrix_element-bounded} hold.
	Then, $v\in C^{0,\alpha}(\Omega)$ for some $0<\alpha<1$.
\end{teo}

\begin{proof}
	Let $v_1$ be a solution to \eqref{eq-Ell-v1}.  
	By Remark \ref{pro-v_2}, $v_2=v-v_1$ solves \eqref{eq-Ell-v2}.
	We prove H{\"o}lder continuity for $v$ by showing it for $v_1$ and $v_2$. 
	
	We split the proof into claims.
	
	{\bf Claim 1:} Estimate for $Dv_1$ in $L^2$
	\begin{equation}\label{teo-main-proof-eq1}
		\int_{B_R}|Dv_1|^2\dx
		\leq
		C \left( R^{d-2+2\alpha}+
		R^\lambda \right)
		\leq 
		C R^\tau, 
	\end{equation}
	which implies $Dv_1\in L^{2,\tau}(\Omega)$, where $\tau=\min\{d-2+2 \alpha, \lambda\}$.
	
	First, we obtain a uniform estimate for the gradient of $v_1$. 
	For that, we estimate the right-hand side in \eqref{lem-est-Cacciopp-state-eq1}.  Since $f_2\in (L^{2,\lambda}(\Omega))^d$ with $d-2<\lambda<d$, by Theorem \ref{teo-divf-Holder}, we have $v_1\in C^{0,\alpha}(\Omega)$ for some $0<\alpha<1$. 
	Thus, in any ball with radius $R$, we have
	$$
	|v_1-\bar{v}_{1}|\leq C R^{\alpha}.
	$$
	Using this and recalling that  $f_2\in (L^{2,\lambda}(\Omega))^d$  from \eqref{lem-est-Cacciopp-state-eq1}, we deduce the \textbf{Claim 1}.
	
	{\bf Claim 2:} The following estimate for $v_2$ holds true
	\begin{equation}\label{key-1}
		\|v_2\|_{L^{1,2+\nu_0}(\Omega)}\leq C,
	\end{equation}
	where $\nu_0=\min\{\lambda,\frac{\lambda}{2},\frac{d+2}{2}\}=\frac{\lambda}{2}$.

	To establish the claim, we  set
	\begin{equation*}%\label{def-f_0}
		f_0(x)=-f_1(x) v(x)-f_3(x)Dv(x)-f_4(x).
	\end{equation*}
	Note that Remark \ref{pro-v_2}, yields
	$\displaystyle{
		-\div(A Dv_2^T)=f_0(x)
	}$. Next, we estimate all terms of $f_0$. Lemma \ref{lem2} and embedding properties of Morrey spaces imply $v \in L^{2,2}(\Omega)$ 
	and $f_3 Dv \in L^{1,\frac{\lambda}{2}}(\Omega)$. Using these
	by Corollary \ref{cor-teo-f-<d-2}, we obtain  \eqref{key-1}.
	
	Because $v_2\in H^1(\Omega)$, Lemma \ref{lem-est-Sobolev} implies
	$v_2\in L^{2^*}(\Omega)$.  This estimate combined with Lemma \ref{lem3}, yields
	\begin{equation}\label{teo-main-proof-eq2}
		\|v_2\|_{L^{2,2+\tau_0}(\Omega)}\leq C,
	\end{equation}
	where
	$\tau_0=\frac{2(2\nu_0-(d-2))}{d+2}$. 
	Note that $\tau_0$ is strictly positive because $\nu_0=\frac{\lambda}{2}>\frac{d-2}{2}$.
	
	{\bf Claim 3} For the gradient of $v_2$, we have 
	\begin{equation}\label{teo-main-proof-eq4}
		\int_{B_R}|Dv_2|^2\dx\leq C R^{\rho_1},
	\end{equation}
	where
	$\rho_1=\min\{\tau_0,\mu_0,d-2+2\alpha\}$. 
	
	To establish the claim, we note that
	$f_3\in (L^{2,\lambda}(\Omega))^d$, $|f_4|^\frac{1}{2}\in L^{2,\lambda}(\Omega)$,  $v_1,v_2\in L^{2,2}(\Omega)$ and $Dv_1,Dv_2\in L^{2}(\Omega)$.
	By using 
	Lemma \ref{lem1},  we estimate last term in \eqref{lem-est-Cacciopp-state-eq2} as follows
	\begin{equation}\label{teo-main-proof-eq3}
		\int_{B_{2R}}(|f_3|^2+|f_4|)(|v_1|^2+|v_2|^2)\dx\leq C R^{\mu_0},
	\end{equation}
	where $\mu_0=\lambda-d+2$. 
	Using   \eqref{teo-main-proof-eq1}, \eqref{teo-main-proof-eq2} and \eqref{teo-main-proof-eq3} in \eqref{lem-est-Cacciopp-state-eq2}, we 
	deduce 
	\begin{equation}\label{lem-est-Cacciopp-state-eq2-rep}
		\int_{B_R}|Dv_2|^2\dx
		\leq 
		C \left(
		R^{d-2+2\alpha}
		+
		R^{d-2}+R^{\tau_0} 
		+ R^{\mu_0}\right),
	\end{equation}
	which completes the proof.

	{\bf Claim 4} For $v$ the following estimate holds
	\begin{equation}\label{teo-main-proof-eq5}
		\|v\|_ {L^{2,2+\rho_1}(\Omega)}
		+
		\|Dv\|_ {L^{2,\rho_1}(\Omega)}\leq C.
	\end{equation}

	By \eqref{teo-main-proof-eq4} in Claim 3, from Lemma \ref{lem-est-Sobolev}, we deduce that $v_2\in L^{2,2+\rho_1}(\Omega)$. Using this and utilizing the following facts $v_1\in L^{2,d}(\Omega)$, $Dv_1\in L^{2,d-2+\alpha}(\Omega)$ and $Dv_2\in L^{2,\rho_1}(\Omega)$, we obtain \eqref{teo-main-proof-eq5}.

	Next, using  \textbf{Claims 1-4}, we improve estimate (see \textbf{Claim 3}) of
$\displaystyle{
		\int_{B_R} |Dv_2|^2\,dx\,.
	}$
	By the Caccioppoli inequality in \eqref{lem-est-Cacciopp-state-eq2}, we have
\begin{equation}\label{ka}
\begin{split}
	\int_{B_R} |Dv_2|^2\,dx
\leq&
\int_{B_R} |Dv_1|^2\,dx
+
\dfrac{1}{R^2} \int_{B_{2R}} (|v_1|^2 + |v_2|^2)\,dx\\
&+
\int_{B_{2R}} (|f_3|^2 + |f_4|) (|v_1|^2 + |v_2|^2)\,dx.
\end{split}
\end{equation}
	By \textbf{Claim 1}, H\"older and Sobolev inequalities, we have 
	$\displaystyle{
	\int_{B_{2R}} |v_1|^2\,dx \leq C R^{\tau + 2}
}.$
	While, by \eqref{teo-main-proof-eq3}
	H\"older, Sobolev, and \textbf{Claim 3} it follows that
	$\displaystyle{
	\int_{B_{2R}} |v_2|^2\,dx \leq C R^{\rho_1 + 2}.
}$
	For the terms with $|f_4|$ of the last integral in \eqref{ka} by Fefferman-Poincarè inequality (see \cite{GiuPiePAMS}), we have 
	$$
	\int_{B_{2R}} |f_4|v_1^2\,\dx
	\leq
	C R^{\lambda-d+2} \int_{B_{2R}} |Dv_1|^2\,\dx,
	\leq 
	C R^{\lambda-d+2 +\tau}
	$$
	and 
	$$
	\int_{B_{2R}} |f_4| v_2^2\,\dx
	\leq
	C R^{\lambda-d+2} \int_{B_{2R}} |Dv_2|^2\,\dx.
	$$
	In the same way, we proceed with $|f_3|^2$ in place of $|f_4|$.
	Combining the  preceding estimates, we get
	\begin{equation*}
		\int_{B_R} |Dv_2|^2\,dx
		\leq
		C
		\left( 
		R^\tau + R^{\rho_1}
		+
		R^{\lambda-d+2+\tau}
		\right) 
		+
		C R^{\lambda-d+2} \int_{B_{2R}} |Dv_2|^2\,dx.
	\end{equation*}
Finally, for sufficiently small $R$ (for more details see e.g. algebraic Lemma in \cite{Stampacchia1964quationsED}), we obtain 
	\begin{equation*}
		\int_{B_R} |Dv_2|^2\,dx
		\leq
		C
		\left( 
		R^\tau + R^{\rho_1}
		+
		R^{\lambda-d+2+\tau},
		\right) 
	\end{equation*}
	and this means $|Dv_2|$ belongs to $L^{2,\chi}(\Omega)$, where
	$\chi=\min(\tau,\rho_1,\lambda-d+2+\tau)$. 
	Our last obstacle is that, in general, $\chi<d-2$ and consequently - at this stage - we can not apply Morrey's lemma to get H\"older continuity of $v_2$.
	We overcome this obstacle by increasing the value of the parameter $\chi$ through  iterating  arguments. This enable us to apply Morrey's lemma and conclude the proof.
	The iterating our arguments are done as follows. 
	
	\textbf{Claim 2}, gives 
	$f_0 \in L^{1,\lambda/2+\chi/2}(\Omega)$ instead of $\frac{\lambda}{2}$ and, then, \eqref{key-1} improves to 
	$v_2$ in $L^{1,2+\lambda/2 + \chi/2}(\Omega)$.
	Now, estimate in \eqref{teo-main-proof-eq2}
	becomes
	$v_2$ belongs to $L^{2,\theta}(\Omega)$, where 
	$$
	\theta=
	\dfrac{2(\lambda+\chi)+\nu(d-2)}{d+2}.
	$$
	In similar way, \eqref{teo-main-proof-eq3}
	becomes
\begin{equation}\label{NEW_teo-main-proof-eq3}
	\int_{B_{2R}}(|f_3|^2+|f_4|)(|v_1|^2+|v_2|^2)\dx\leq C R^{\mu_0},
\end{equation}
	where the new exponent is $\mu_1=\lambda-d+2 + \chi/2$. 
	As a consequence
	$Dv_2$ belongs to $L^{2,\varrho_2}(\Omega)$, where $\rho_2=\rho_1+\chi/2$. Note that $\rho_2$ is strictly greater than $\rho_1$.
	By a finite number of steps, we get $Dv_2$ in the same space as $Dv_1$, and the result follows.
\end{proof}		

\subsection{Further  H\"older regularity }
In Proposition \ref{pro-C^a-reg}, we proved existence of $C^{0,\alpha}$  solutions to \eqref{eq-Ell1}.
Here, we prove $C^{1,\alpha}$  H\"older regularity. For that, a crucial step is to prove that the gradient of solutions 
belongs to a suitable Morrey space. Then,  the regularity result for the linear elliptic equations in Theorem \ref{teo-main-Holder} leads to H\"older continuity for the gradient.

By Theorem \ref{teo-ex-Linf-sol-quad} and  Proposition \ref{pro-C^a-reg}
there exists $u \in H^1(\Tt^d)\cap C^{0,\alpha}(\Tt^d)$ 
solving  \eqref{eq-Ell1} in the sense of distributions. 
Next, we prove an estimate for the gradient by a Caccioppoli inequality.
% for the solution to \eqref{eq-ge-el}. 
\begin{proposition}\label{pro-gen-Caccioppoli}
	Let \(u\in H^1(\Tt^d)\cap C^{0,\alpha}(\Tt^d)\) solve \eqref{eq-Ell1} and suppose that Assumptions \ref{assump-A_matrix-eliptic}-\ref{assump-V-g_bounded} hold. 
	Then, there exists  a positive constant \(C\) depending only on  data  such that  \(\|Du\|_{L^{2,\lambda}(\Tt^d)}\leq C\) for some \(d-2<\lambda<d\).
\end{proposition}

\begin{proof} 
	By definition of weak solution to \eqref{eq-Ell1}, we have		
	\begin{equation}\label{pro-gen-Cacci-weak}
		\int_{\Tt^d} DuA(x) D\varphi^T\dx=-\frac{1}{2}	\int_{\Tt^d}DuA(x)Du^T\varphi\dx+\int_{\Tt^d} (g(-u)-V_{\Hh}) \varphi\dx,
	\end{equation} 
	for all \(\varphi\in  H^1(\Tt^d)\cap L^\infty(\Tt^d)\). 
	
	Let  \(0<R<R_0\), where \(R_0\) will be chosen  later.
	Because \(u\in H^1_0(\Tt^d)\cap C^{0,\alpha}(\Tt^d)\), we have
	\begin{equation}\label{eq-w_minus_w_2r}
		|u-u_{2R}|\leq C R^{\alpha},
	\end{equation} 
	where \(u_{2R}=\frac{1}{|B(2R)|}\int_{B(2R)}u~\dx\). 
	Now, let
	\[
	\varphi(x)=\eta^2(x)(u(x)-u_{2R}),
	\] 
	where \(\eta\) is supported in $B_{2R}$, identically $1$ in $B_R$ and is a smooth cut-off function.
	We observe that  \(\varphi\in H^1_0(\Tt^d)\cap C^{0,\alpha}(\Tt^d)\) and 
	\begin{equation}\label{def-varphi_diff}
		D\varphi=2\eta(u-u_{2R})D\eta +\eta^2Du.
	\end{equation} 
	By \eqref{def-varphi_diff} and \eqref{pro-gen-Cacci-weak}, we obtain 
	\begin{align}\label{pro-gen-proof_eq1}
		\int_{B(2R)} \eta^2DuA(x) Du^T\dx
		=
		&-	2\int_{B(2R)} \eta(u-u_{2R})Du A(x)D\eta^T \dx
		\\
		\nonumber
		&
		-\frac{1}{2} \int_{B(2R)} DuA(x)Du^T \eta^2(u-u_{2R})\dx
		\\
		\nonumber
		&
		+	\int_{B(2R)} (g(-u)-V_\Hh) \eta^2(u-u_{2R})\dx\,.
	\end{align}  
	From Assumption \ref{assump-A_matrix-eliptic}, we have
	\begin{equation}\label{def-A_max}
		\max_{x\in\Tt^d} |a_{ij}(x)|\leq \theta_1.
	\end{equation} 
	Thus,
	\begin{equation}\label{pro-gen-proof_est1}
		2\left| \int_{B(2R)} \eta(u-u_{2R})Du A(x)D\eta^T \dx\right| \leq \frac{\theta_0}{2}\int_{B(2R)} \eta^2|Du|^2  \dx+C \int_{B(2R)} |u-u_{2R}|^2|D\eta|^2 \dx.
	\end{equation}  
	By \eqref{def-A_max}, we get
	\begin{equation}\label{pro-gen-proof_est2}
		\frac{1}{2}\Big| \int_{B(2R)} DuA(x)Du^T \eta^2(u-u_{2R})\dx\Big| 
		\leq \frac{\theta_1}{2}\int_{B(2R)} \eta^2|Du|^2  |u-u_{2R}|\dx.
	\end{equation}  
	Again using Assumption \ref{assump-A_matrix-eliptic}, taking into account the estimates in \eqref{pro-gen-proof_est1}, \eqref{pro-gen-proof_est2},  \eqref{eq-w_minus_w_2r},  \eqref{pro-gen-proof_eq1} and recalling that \(|D\eta|\leq \frac{1}{R}\) (follows from the definition of $\eta$),  we deduce
	\[
	\theta_0\int_{B(2R)} \eta^2|Du|^2  \dx
	\leq \frac{\theta_0}{2}\int_{B(2R)} \eta^2|Du|^2  \dx +CR^{d-2+2\alpha}+	\frac{\theta_1}{2}R^{\alpha}\int_{B(2R)} \eta^2|Du|^2  \dx + CR^{d+\alpha}.
	\] 
	Then, fixing  \(R_0>0\)  such that \(\frac{\theta_1}{2}R^{\alpha}\leq\frac{\theta_0}{4}\) for all \(0<R\leq R_0\) and  using the proceeding inequality, we get 
	\[
	\begin{split}
		\int_{B(R)} |Du|^2  \dx\leq C(R^{d-2+2\alpha}+	R^{d+\alpha})\leq CR^{d-2+2\alpha},
	\end{split}
	\] 
	i.e. \(Du\in( L^{2,\lambda}(\Tt^d))^d\) for \(\lambda=d-2+2\alpha\).
\end{proof}

Next, relying on the preceding proposition and the results from Section \ref{sec:4}, we prove that the solution to  \eqref{eq-Ell1} is continuously differentiable.

\begin{proposition}\label{pro-C^1-reg}
 Let \(u\in H^1(\Tt^d)\cap L^\infty(\Tt^d)\) be a solution to  \eqref{eq-Ell1}	and suppose that Assumptions \ref{assump-A_matrix-eliptic}-\ref{assump-A_matrix_element-bounded} hold . 
	Then, \(u\in H^1(\Tt^d)\cap C^{1,\gamma}(\Tt^d)\) for some \(\gamma>0\).
\end{proposition}

%
%{\color{red} So, we proved that there exists \(u_\eta \in H^1(\Tt^d)\cap L^\infty(\Tt^d)\) solving  \eqref{eq-Ell1} in the sense of distributions. Now, by Theorem 1.1 in Chapter 4 of \cite{Lad} follows that any generalized solution $u\in H^1(\Tt^d)\cup L^\infty(\Tt^d)$ to \eqref{eq-Ell1} is H{\"o}lder continuous; that is, there exists a $0<\alpha<1$ such that $u\in H^1(\Tt^d)\cup C^{0,\alpha}(\Tt^d)$. 
	%}

\begin{proof}
Differentiating \eqref{eq-Ell1} with respect to \(x_k\), \(k\in\{1,\dots,d\}\) and setting \(u_{x_k}=v\), we obtain
	\begin{align*}
		-\div(ADv^T )
		&
		-\div(A_{x_k}Du^T )+\frac{1}{2}DvA Du^T+\frac{1}{2}DuA Dv^T+
		\\
		&
		+\frac{1}{2}DuA_{x_k} Du^T+	g^\prime(-u)v+(V_\Hh)_{x_k}=0.
	\end{align*}
	We rewrite the preceding equation as
	\begin{equation*}%\label{eq-HJ-diff-grad}
		-\div(ADv^T )+f_1 v+\div(f_2)+f_3Dv+f_4=0,
	\end{equation*} 
	where 
	\begin{equation*}%\label{def-fs}
		f_1=g^\prime(-u), \quad f_2=A_{x_k}Du^T, \quad
		f_3=\frac{1}{2}\left( Du A+ DuA^T\right), \quad 
		f_4=Du A_{x_k}Du^T+V_{x_k}.
	\end{equation*} 
	Because \(u\in H^1(\Tt^d)\cap C^{0,\alpha}(\Tt^d)\), and taking into account  Proposition \ref{pro-gen-Caccioppoli}, we have
	\begin{equation}\label{rem-assm-on-fs} 
		f_1\in C^{0,\alpha}(\Tt^d),\quad f_2,f_3\in (L^{2,\lambda}(\Tt^d))^d,\quad\text{and}\quad f_4\in L^{1,\lambda}(\Tt^d),
	\end{equation}
	for some \(\lambda\in(d-2,d)\).  We complete the proof by showing $v\in C^{0,\alpha}(\Tt^d)$. Similar to the previous case, we denote by \(v^*\) the periodic extension of \(v\) to \(\Rr^d\). 
	Then,
	\begin{equation}\label{proof-eq-main1}
		\int_{\Rr^d}D\varphi A D(v^*)^T-D\varphi f_2+(f_1v^*+f_3Dv^*+ f_4)\varphi\, \dx=0,
	\end{equation}
	for all \(\varphi\in  H_0^1(\Rr^d)\). As discussed before, it is enough to prove that
	$v^*\in  C^{0,\alpha}(\Omega_2)$ for some $\Omega_2\subset\Rr^d$ satisfying  \(\Omega_1= [-1,1]^d\subset\Omega_2\). Let $\Omega_1$ and $\Omega_2$ be as before and  let $\eta_2$ is defined by \eqref{def-zeta-2}  (a smooth standard cut-off identically 1 on $\Omega_1$ and compactly supported in $\Omega_2$).  Using $\eta_2$, next, we construct a  boundary value problem for the function, $v_\eta=\eta_2v^*$,  and use the regularity results proved in Section \ref{sec:5} to obtain  H{\"o}lder regularity of $v_\eta$,  therefore, H{\"o}lder regularity of $v^*$.
	Taking $\varphi(x)=\psi(x)\eta_2(x)$, $\psi\in H^1(\Rr^d)$ as a test function  in \eqref{proof-eq-main1}, we get
	\begin{equation}\label{proof-eq-main2}
		\begin{split}
			\int_{\Rr^d}D\psi A D(\eta_2 v^*)^T&-D\psi  A ( v^*D\eta_2)^T+\psi D\eta_2 A (Dv^*)^T-\psi(D\eta_2  f_2)-D\psi (\eta_2f_2) \\&+\psi f_3 D(\eta_2v^*)-\psi f_3v^*D\eta_2+(f_1\eta_2v^*+ \eta_2f_4)\psi\, \dx=0.
		\end{split}
	\end{equation}
	Similar to \eqref{Ho-eq-main3}, we have
	\begin{equation}\label{proof-eq-main3}
		\begin{split}
			\int_{\Rr^d}\psi D\eta_2 A (Dv^*)^T\, \dx=-\int_{\Rr^d}v^*D\psi( D\eta_2 A )^T+\psi v^*\div(D\eta_2 A ) \, \dx.
		\end{split}
	\end{equation}
	Using \eqref{proof-eq-main3} in \eqref{proof-eq-main2} and rearranging terms, we deduce that
	\begin{equation*}
		\begin{split}
			\int_{\Rr^d}&D\psi A D(\eta_2 v^*)^T-D\psi ( v^*A (D\eta_2)^T+\eta_2f_2+v^*(D\eta_2 A)^T )+\psi f_3 D(\eta_2v^*) \\&-( D\eta_2  f_2+ f_3v^*D\eta_2+v^*\div(D\eta_2 A )-f_1(\eta_2v^*)- \eta_2f_4)\psi\, \dx=0.
		\end{split}
	\end{equation*}
	Recalling that $v_\eta=\eta_2v^*$, we notice that $v_\eta$ solves 
	\begin{equation}\label{eq-boundary}
		\begin{cases}
			-\div(A Dv_\eta^T)+f_1v_\eta+\div(\tilde{f}_2) +f_3Dv_\eta+\tilde{f}_4 =0	&\quad \text{in}\quad\Omega_2	\\
			v_\eta=0	&	\quad \text{on}\quad\partial\Omega_2
		\end{cases}
	\end{equation}
	where 
	\begin{equation*}
		\begin{split}
			&\tilde{f}_2=v^*A (D\eta_2)^T+f_2\eta_2+v^*(D\eta_2 A)^T,\\ &\tilde{f}_4= -v^*\div(D\eta_2A )-f_2D\eta_2-f_3v^*D\eta_2+f_4\eta_2.
		\end{split}
	\end{equation*}
	By Proposition \ref{pro-gen-Caccioppoli}, we have  \(v=u_{x_k}\in L^{2,\lambda}(\Tt^d)\) with \(\lambda=d-2+2\alpha\). Therefore, \(v^*\in L^{2,\lambda}(\Omega_2)\). Hence,
	combining the Lipschitz continuity of $A$ with \eqref{rem-assm-on-fs}, we obtain
	\begin{equation*} 
		f_1\in C^{0,\alpha}(\Tt^d),\quad \tilde{f}_2,f_3\in (L^{2,\lambda}(\Tt^d))^d,\quad \tilde{f}_4\in L^{1,\lambda}(\Tt^d).
	\end{equation*}
		%	The condition $\lambda>d-2$ is crucial for  the H\"older continuity, see  Section \ref{sec:5}.  
		Consequently,  because $\lambda>d-2$, we apply Theorem \ref{teo-main-Holder} to \eqref{eq-boundary} and deduce  that $v_\eta \in C^{0,\alpha}(\Omega_2)$. Using this and recalling that $\eta_2$ is smooth with $\eta_2(x)=1$ for all $x\in\Omega_1$, we obtain  $v^* \in C^{0,\alpha}(\Omega_1)$ and this  completes the proof. 
\end{proof}

\subsection{Uniqueness of solutions}

Finally, we address the uniqueness of $C^{1, \alpha}$ solutions to \eqref{eq-Ell-bound}. We begin by proving a lemma. 
\begin{lem}
	\label{mprin}
	Suppose that Assumptions \ref{assump-A_matrix-eliptic} and \ref{assump-A_matrix_element-bounded} hold. 
	Let $u\in H^1(\Tt^d)\cap C^{1,\alpha}(\Tt^d)$ solve
	\[
	\div(Du(x) A(x)) =f(x),
	\]
	where $f$ is a continuous function. Suppose that $u$ has a maximum at a point $x_M$. Then,
	\[
	f(x_M)\leq 0.
	\]
	Similarly, $f(x_m)\geq 0$ for a point of minimum, $x_m$, of $u$. 
\end{lem}
\begin{proof}
	Let $x_M$ be a point of maximum of $u$. Without loss of generality, 
	we assume that $x_M=0$. Because $u\in C^{1,\alpha}(\Tt^d)$,  $Du(0)=0$. 
	Let $\bar A=A(0)$ and  $v(y)=u( \bar A^{1/2}y)$.	Let $B_\tau$ denote the ball centered at 
	the origin with radius $\tau$ and $0<\tau <r$, for some sufficiently small $r$. Finally, 
	we let $\nu$ be the outer unit normal to $\partial (\bar A^{1/2}B_\tau)$.
	Then,	
	\begin{align*}
		\int_{\bar A^{1/2}B_\tau} \div(Du(x) A(x))dx&=\int_{\bar A^{1/2}B_\tau}  \div(Du(x) (A(x) -\bar A))+\div(Du(x) \bar A) dx\\
		&=\int_{\partial (\bar A^{1/2}B_\tau)}  Du(x) (A(x) -\bar A)\nu+\int_{\bar A^{1/2}B_\tau} \div(Du(x) \bar A) dx\\
		&=O(\tau^{d+\alpha}) + \int_{A^{1/2}B_\tau} u_{x_i x_j}(x) \bar A_{i j} dx\\
		&=O(\tau^{d+\alpha}) + c_{\bar A} \int_{B_\tau} u_{x_i x_j} (\bar A^{1/2}y) \bar A_{i j} dy\\
		&=O(\tau^{d+\alpha}) + c_{\bar A} \int_{B_\tau} \div(D v)dy\\
		&=O(\tau^{d+\alpha}) + c_{\bar A} \int_{\partial B_\tau}  Dv(y) \frac{y}{\tau} dS(y)\\
		&=O(\tau^{d+\alpha}) + c_{\bar A} \int_{\partial B_\tau}  Du(\bar A^{1/2}y)  \bar A^{1/2}\frac{y}{\tau}
		dS(y)\\
		&=O(\tau^{d+\alpha}) + c_{\bar A}\tau^{d-1} \int_{\partial B_1}  Du(\tau \bar A^{1/2}z)  \bar A^{1/2}z
		dS(z)\\
		&=O(\tau^{d+\alpha}) + c_{\bar A}\tau^{d-1} \frac{d}{d\tau} \int_{\partial B_1}  u(\tau \bar A^{1/2}z) dS(z).
	\end{align*}	
	Therefore, we have
	\begin{align*}
		\frac{\tau }{|\bar A^{1/2}B_\tau|}\int_{\bar A^{1/2}B_\tau} \div(Du(x) A(x))dx&
		&=O(\tau^{1+\alpha}) + \tilde c_{\bar A} \frac{d}{d\tau} \int_{\partial B_1}  u(\tau \bar A^{1/2}z) dS(z).
	\end{align*}	
	Integrating from $0$ to $r$, 
	and
	taking into account that 
	\[
	\int_{\partial B_1}  \left(u(r\bar A^{1/2}z)-u(0)\right)\d S(z)\leq 0,
	\]
	we obtain
	\[
	\int_0^r	\frac{\tau }{|\bar A^{1/2}B_\tau|}\int_{\bar A^{1/2}B_\tau} \div(Du(x) A(x))\dx
	\leq O(r^{2+\alpha}).
	\]
	From which it follows
\begin{equation}\label{lala}
	\limsup_{r\to 0}
\frac 1 {r^2} \int_0^r	\frac{\tau}{|\bar A^{1/2}B_\tau|}\int_{\bar A^{1/2}B_\tau} \div(Du(x) A(x))\dx
\leq 0.
\end{equation}
	Noticing that
\begin{equation}\label{laeq}
	\frac 1 {r^2} \int_0^r	\frac{\tau}{|\bar A^{1/2}B_\tau|}\int_{\bar A^{1/2}B_\tau} \div(Du(x) A(x))\dx
=\frac 1 {r^2} \int_0^r	\frac{\tau}{|\bar A^{1/2}B_\tau|}\int_{\bar A^{1/2}B_\tau} F(x) \dx.
\end{equation}
Recalling that $f$ is continuous for some $\hat c>0$, we have 
	\[
\frac 1 {r^2} \int_0^r	\frac{\tau}{|\bar A^{1/2}B_\tau|}\int_{\bar A^{1/2}B_\tau} F(x) dx\to \hat c F(0).
\]
	Using this and \eqref{laeq} in \eqref{lala},  we establish the first part of the lemma. The case of a minimum is analogous.
\end{proof}

\begin{pro}
	\label{pun}
	Suppose that Assumptions \ref{assump-A_matrix-eliptic}-\ref{assump-A_matrix_element-bounded} hold.
	Then, there exists at most one solution $u\in H^1(\Tt^d)\cap C^{1,\alpha}(\Tt^d)$
	to  \eqref{eq-Ell1}. 	
\end{pro}
\begin{proof}
	Suppose there are two solutions to  \eqref{eq-Ell1}, $u$ and $\tilde u$. Let $x_M$ be a point 
	of maximum of $u-\tilde u$. Note that $Du(x_M)=D\tilde u(x_M)$. 
	Then, the preceding Lemma implies
	\[
	-g(-u(x_M))+g(-\tilde u(x_M))\leq 0.
	\]
	By the monotonicity of $g$, $u(x_M)-\tilde u(x_M)\leq 0$. Thus, $u\leq \tilde u$.
	Exchanging the roles of $u$ and $\tilde u$, we conclude that $u=\tilde u$. 
\end{proof}

\begin{cor}
	\label{stable}
	Suppose that Assumptions \ref{assump-A_matrix-eliptic}-\ref{assump-A_matrix_element-bounded} hold.
	Then, $\Hh\mapsto u_{\Hh}$ is continuous as a map from $\Rr$ to $C^{1,\alpha}(\Tt^d)$.
\end{cor}
\begin{proof}
	Consider a sequence $\Hh_n$ converging to some value $\Hh$. Let $ u_{\Hh_n}$ be corresponding solutions to \eqref{eq-Ell1} with $\Hh_n$.  We claim that $ u_{\Hh_n}$
	converges to $u_{\Hh}$, where $u_{\Hh}$ is the solution to \eqref{eq-Ell1} for $\Hh$. Arguing as in the proof of Theorem \ref{teo-ex-Linf-sol-quad}, we deduce that any subsequence of $u_{\Hh_n}$
	converges to a solution of \eqref{eq-Ell1}. Because of the uniqueness result in Proposition \ref{pun}, 
	this limit is unique. Hence, the whole sequence converges. 
\end{proof}

\section{Existence and Uniqueness of Solutions to the Stationary MFG}
\label{sec:6}

This section discusses the existence and uniqueness of solutions to the stationary MFG problem given in Problem \ref{problem-1} proving Theorem \ref{teo-main}.

\subsection{Proof of Theorem \ref{teo-main}}

Based on previous sections' results, for any $\Hh$, there exists a  solution to Problem \ref{problem-1}. However, this solution may fail to satisfy the normalization condition $\int_{\Tt^d}m\dx=1$. 

In this section, we show the existence of a constant $\Hh$ such that there exists a normalized solution to Problem \ref{problem-1}.
This finalizes proof of our main result, Theorem \ref{teo-main}.

\begin{proof}[Proof of Theorem \ref{teo-main}]
	First, we prove the uniqueness for Problem \ref{problem-1}. This follows the standard Lasy-Lions monotonicity arguments.
	
	Let  \((u_1,m_1,\bar{H}_1),(u_2,m_2,\bar{H}_2)\in C^1(\Tt^d)\times C^1(\Tt^d)\times\Rr\)  solve Problem \ref{problem-1}. That is, for \(i=1,2\) and for any \(\varphi\in H^1(\Tt^d) \),
	\begin{equation}\label{proof-pro_unique-eq1}
		\begin{cases}
			\int_{\Tt^d}D\varphi A Du_i^T +\left( \frac{1}{2}Du_iADu_i^T +V(x)+\bar{H}_i\right) \varphi\dx=\int_{\Tt^d}g\left( \log m_i\right)\varphi\dx\\
			\int_{\Tt^d}D\varphi A Dm_i^T  +m_i D\varphi  A Du_i^T  \dx=0.
		\end{cases}
	\end{equation}	
	Because \(m_1,m_2\) are probability density functions, we have
	\[
	\int_{\Tt^d}(m_1-m_2) \bar{H}_i\dx=0.
	\]
	Therefore, recalling that  \(m_1,m_2\in C^1(\Tt^d)\), we take \((m_1-m_2)\) as a test function in the first equation in \eqref{proof-pro_unique-eq1} and subtract the corresponding equations for \(i=1,2\) to get
	\begin{equation}\label{proof-pro_unique-eq3}
		\begin{split}
			\int_{\Tt^d}(Dm_1-Dm_2) A (Du_1^T-Du_2^T) &+\frac{(m_1-m_2)}{2}\left( Du_1ADu_1^T-Du_2ADu_2^T\right)\dx\\
			&=\int_{\Tt^d}(g\left( \log m_1\right)-g\left( \log m_2\right))(m_1-m_2)\dx.
		\end{split}
	\end{equation}	
	We argue similarly for $u_1,u_2\in C^1(\Tt^d)$. Taking \((u_1-u_2)\) as a test function in the second equation in \eqref{proof-pro_unique-eq1} and subtracting the corresponding equations for \(i=1,2\), we obtain
	\begin{equation}\label{proof-pro_unique-eq4}
		\begin{split}
			\int_{\Tt^d}(Du_1-Du_2) A (Dm_1^T-Dm_2^T) +(Du_1-Du_2)\left( m_1ADu_1^T-m_2ADu_2^T\right)\dx=0.
		\end{split}
	\end{equation}	
	From \eqref{proof-pro_unique-eq3} and \eqref{proof-pro_unique-eq4}, we get
	\begin{equation*}%\label{proof-pro_unique-eq5}
		\begin{split}
			\int_{\Tt^d}\frac{(m_1+m_2) }{2}(Du_1-Du_2) &A (Du_1^T-Du_2^T) \\
			&+(g\left( \log m_1\right)-g\left( \log m_2\right))(m_1-m_2)\dx=0.
		\end{split}
	\end{equation*}	
	Assumptions \ref{assump-A_matrix-eliptic} and \ref{assump-g_mon-concave}, yield  \(Du_1=Du_2\) and \(m_1=m_2\).  Using these in the first equation in \eqref{proof-pro_unique-eq1} with the test function \(\varphi\equiv 1\) for \(i=1,2\), we deduce  that \(\bar{H}_1=\bar{H}_2\). Therefore, there exists at most one triple \((u,m, \bar{H})\in C^1(\Tt^d)	\times C^1(\Tt^d)\times\Rr\) that solves Problem 
	\ref{problem-1}. 
	
	Next, we address the existence part. 
	Fix \(\bar{H}\in \Rr\). Propositions \ref{teo-ex-Linf-sol-quad} and \ref{pro-C^1-reg} imply the existence of a function \(u_{\bar{H}}\in C^1(\Tt^d)\) solving \eqref{eq-HJ_Geop-after-Hopf}.
	Let \(m_{\bar{H}}=e^{-u_{\bar{H}}}\in C^1(\Tt^d)\).  We note that  \((u_{\bar{H}},m_{\bar{H},}\bar{H})\) satisfies 
	\begin{equation*}%\label{proof-teo-main-eq1}
		\begin{cases}
			\int_{\Tt^d}D\varphi A Du_{\bar{H}}^T +\left( \frac{1}{2}Du_{\bar{H}}ADu_{\bar{H}}^T +V(x)-\bar{H}\right) \varphi\dx=\int_{\Tt^d}g\left( \log m_{\bar{H}}\right)\varphi\dx\\
			\int_{\Tt^d}D\varphi A Dm_{\bar{H}}^T  +m_{\bar{H}} D\varphi  A Du_{\bar{H}}^T  \dx=0,
		\end{cases}
	\end{equation*}
	for all \(\varphi\in H^1(\Tt^d)\). 
	Accordingly, by  Definition \ref{def-weak}, 
	it is  enough to prove that there exists a constant \(\bar{H}\) such that
	\(m_{\bar{H}}=e^{-u_{\bar{H}}}\)  
	satisfies
	\begin{equation*}%\label{proof-teo-main-eq2}
		\int_{\Tt^d} m_{\bar{H}}\dx=\int_{\Tt^d}e^{-u_{\bar{H}}}\dx=1.
	\end{equation*}
	To do so, we define \(\mathcal{H}:\Rr\to\Rr\) by
	\begin{equation*}%\label{def-H-func}
		\mathcal{H}(\bar{H})=\int_{\Tt^d}e^{-u_{\bar{H}}}\dx.
	\end{equation*}
	By Corollary \ref{stable}, we have that \(\mathcal{H}\) is continuous. 
	
	To complete the proof, we show that there exist constants \(\bar{H}_{low}\)  and  \(\bar{H}_{up}\) such that 
	\begin{equation*}
		\mathcal{H}(\bar{H}_{low})<1,\quad \mathcal{H}(\bar{H}_{up})>1.
	\end{equation*}
	We begin by proving the existence of $\bar{H}_{up}$.
	We recall that 
	\(u_{\bar{H}}\in C^{1,\alpha}(\Tt^d)\) satisfies
	\begin{equation*}%\label{proof-teo-main-eq4}
		-\div (A Du_{\bar{H}}^T )+ \frac{1}{2}Du_{\bar{H}}ADu_{\bar{H}}^T +V(x)-\bar{H}-g\left( -u_{\bar{H}}\right)=0.
	\end{equation*}
	Consider a maximum point $x_M$  for \(u_{\bar{H}}\). Note that $Du_{\bar{H}}(x_M)$=0.
	By Lemma \ref{mprin}, 
	we get 
	\[
	V(x_M)-\bar{H}-g\left( -u_{\bar{H}}(x_M)\right)\leq 0.
	\]
	Therefore, 
	\begin{equation*}
		g(-\max u_{\bar{H}})\geq C(-\|V\|_{\infty}-\bar H).
	\end{equation*}
	Taking  
	$\bar{H}_{up}<-\|V\|_{\infty}-C-\frac{C^{-1}}{C_g}$ in proceeding inequality, we get
	\begin{equation}\label{main-l}
		g(-\max u_{\bar{H}_{up}})>\frac{1}{C_g}.
	\end{equation}
Notice that by Assumption \ref{assump-V-g_bounded},  for all $q\geq 0$,  we have $g(-q)\leq\frac{1}{C_g}$. Furthermore,   Assumption  \ref{assump-V-g_bounded} implies that there exists $q_0>0$ such that $g(q)>\frac{1}{C_g}$ if and only if $q\geq q_0>0$. This with \eqref{main-l},  yields  $\max u_{\bar{H}_{up}}<0$. Hence, 
	\[
	\mathcal{H}(\bar{H}_{up})\geq e^{-\max u_{\bar{H}_{up}}}>1.
	\]
	
	Next, we prove the existence of  $\bar{H}_{low}$.  Arguing as in the case of  $\bar{H}_{up}$, we conclude that
	\begin{equation}\label{pre-last}
		g(-\min u_{\bar{H}}) \leq C(\|V\|_{\infty}-\bar H-C).
	\end{equation}
	Similar to the previous case,  Assumption \ref{assump-V-g_bounded} implies that there exists $q_0<0$ such that  $g(q)<-\frac{1}{C_g}$ if and only if $q<q_0$. This with   \eqref{pre-last} implies that by  taking $\bar{H}_{low}>\|V\|_{\infty}-C+\frac{C^{-1}}{C_g}$, we get $-\min \bar{H}_{low}<0$. Thus,
	\[
	\mathcal{H}(\bar{H}_{low})\leq e^{-\min u_{\bar{H}_{low}}}<1.\qedhere
	\]
\end{proof}

\bibliographystyle{plain}
\IfFileExists{"mfgTig.bib"}{\bibliography{mfgTig.bib}}

\end{document}